\documentclass[11pt]{amsart}
\usepackage{amsmath,amssymb,mathrsfs,color}
\usepackage[title]{appendix}
\usepackage{amssymb}
\usepackage{enumerate}
\allowdisplaybreaks

\let\cal=\mathcal
\def\N{{\mathbb N}}
\def\Z{{\mathbb Z}}
\def\R{{\mathbb R}}
\def\P{{\mathbb P}}
\def\E{{\mathbb E}}

\def\eps{\epsilon}
\newtheorem{thm}{Theorem}[section]
\newtheorem{cor}[thm]{Corollary}
\newtheorem{lem}[thm]{Lemma}
\newtheorem{prop}[thm]{Proposition}

\theoremstyle{definition}
\newtheorem{de}[thm]{Definition}
\theoremstyle{remark}
\newtheorem{rem}[thm]{Remark}
\newtheorem{exam}[thm]{Example}
\numberwithin{equation}{section}

\newcommand{\rmd}{{\rm d}}

\allowdisplaybreaks

\begin{document}

\title[Wasserstein convergence rates in the quenched invariance principle]{Quenched invariance principle with a rate for random dynamical systems} 

\author{Zhenxin Liu}
\address{Z. Liu: School of Mathematical Sciences, Dalian University of Technology, Dalian
116024, P. R. China}
\email{zxliu@dlut.edu.cn}

\author{Benoit Saussol}
\address{B. Saussol: Aix Marseille Universit\'e, CNRS, I2M, 13009 Marseille, France}
\email{benoit.saussol@univ-amu.fr}

\author{Sandro Vaienti}
\address{S. Vaienti: Universit\'e de Toulon,  Aix Marseille Universit\'e, CNRS, CPT, 13009 Marseille, France.}
\email{vaienti@cpt.univ-mrs.fr}

\author{Zhe Wang}
\address{Z. Wang: School of Mathematical Sciences, Dalian University of Technology, Dalian
116024, P. R. China}
 \email{zwangmath@hotmail.com}

\date{June 15, 2025}



\begin{abstract}
In this paper, we consider the quenched invariance principle for random Young towers driven by an ergodic system. In particular, we obtain the Wassertein convergence rate in the quenched invariance principle.
As a key ingredient,  we derive a new martingale-coboundary decomposition for the random tower map, which provides a good control over sums of squares of the approximating martingale. We apply our results to a class of random dynamical systems that admit a random Young tower, such as independent and identically distributed (i.i.d.) translations of Viana maps, intermittent  maps of the interval and
small random perturbations of Anosov maps with an ergodic driving system.
\end{abstract}

\maketitle

\section{Introduction}
\setcounter{equation}{0}

In this paper, we consider statistical properties for a class of random dynamical systems. In such systems, the evolution can be seen as a random composition of maps acting on the same space $X$, where maps are chosen according to a stationary process. Formally, a random dynamical system can be described by a skew product $T:\Omega\times X\to \Omega\times X$ given by
\[
T(\omega,x):=(\sigma \omega, f_\omega(x)),
\]
where $f_\omega:X\to X$ is a family of measurable maps and $\sigma:\Omega\to \Omega$ is a measure-preserving transformation on a probability space $(\Omega, \cal F, \P)$. Each $f_\omega$ is often called the fiber map and $\sigma$ is called the base map or the driving system. The random orbits are obtained by the iteration of the maps $f_\omega^n=f_{\sigma^{n-1}\omega}\circ \cdots\circ f_{\sigma\omega}\circ f_\omega$. We refer to \cite{A98,K86} for a systematic introduction to these systems, where in \cite{K86}, Kifer considered the i.i.d. case, i.e. the maps are chosen independently with identical distributions and in \cite{A98}, Arnold considered the general case, iterating a stationary stochastic sequence of maps.

Over the past few decades, there has been a remarkable interest in studying statistical limit theorems for random dynamical systems
\cite{ANV15,ABRV23,BBR19,BBMD02,CT23,Du15,H22,LV18,Su22}. Among these results, a random Young tower (RYT), a random extension of the Young tower \cite{Y98,Y99}, is an important tool for studying limit theorems, such as the almost sure mixing rates and the quenched almost sure invariance principle, for random dynamical systems with weak hyperbolicity. The RYT was first constructed in \cite{BBMD02} to obtain an exponential almost sure mixing rate for
i.i.d. translations of unimodal maps. It was then extended to a more general RYT in
\cite{Du15} and applied to a larger class of unimodal maps. In \cite{BBR19}, it was adapted to i.i.d.
perturbations of the intermittent maps introduced in \cite{LSV}  to obtain a polynomial almost sure mixing rate under additional assumptions.
More recently, Alves et al \cite{ABRV23} dropped the i.i.d. assumption and studied the RYT driven by an ergodic automorphism, proving quenched exponential correlations decay for tower maps admitting exponential tails.

Regarding quenched limit theorems, Su \cite{Su22} obtained the quenched almost sure invariance principle (QASIP) for random Young towers in the i.i.d. case. Hafouta~\cite{H22} proved the Berry-Esseen theorem, the local central limit theorem, and several versions of large deviation principles for i.i.d. random non-uniformly expanding or hyperbolic maps with exponential first return times.

In the context of general ergodic driving systems, numerous significant contributions have been made \cite{AFGV21,AFGV23a,AFGV23b,B99,DFGV18,DFGV18b,DFGV20,DH21}. Buzzi \cite{B99} used the Birkhoff cones technique to obtain the quenched exponential decay of correlation for random Lasota-Yorke maps. Under the similar setting, Dragi\v{c}evi\'{c} et al \cite{DFGV18} obtained the QASIP for random piecewise expanding map, asking the family of maps uniformly good. In \cite{DFGV18b,DFGV20,DH21}, the authors extended Gou\"ezel's spectral approach to obtain the (vector-valued) QASIP for certain class of random dynamical systems. Notably, these results we mentioned above do not ask any mixing assumption on the driving system.

In the present paper, we consider the quenched invariance principle for random Young towers driven by an ergodic system. In addition, we obtain the convergence rate in the Wasserstein distance.  For $p\ge 1$, we denote by $\cal W_p(P,Q)$ the Wasserstein distance between the distributions $P$ and $Q$ on a Polish space $(\cal X,d)$ (see \cite[Definition~6.1]{V09}):
\[
\cal W_p(P,Q)= \inf \{ [E {d(X,Y)}^p]^{1/p}; \hbox{law} (X)=P, \hbox{law} (Y)=Q \}.
\]
We refer to \cite[Chapter~6]{V09} for further details on the Wasserstein distance.

Other early works on the convergence rates in the weak invariance principle for deterministic dynamical systems go back to Antoniou and Melbourne \cite{A18,AM19}. More recently, Liu and Wang obtained the Wasserstein convergence rates in the invariance principle for deterministic systems \cite{LW24} and sequential dynamical systems \cite{LWnon}. Paviato \cite{P24} obtained the convergence rates in the multidimensional weak invariance principle for discrete- and continuous-time dynamical systems. Subsequently, Fleming-V\'azquez \cite{F25} derived the Wasserstein convergence rates in the multidimensional weak invariance principle for more general nonuniformly hyperbolic systems.

In random dynamical systems, to estimate the convergence rate in the quenched invariance principle,
we redefine a self-normalized continuous process $\overline W_{n}^{\omega}$ as in \eqref{wnt} and
obtain the Wasserstein convergence rate $O(n^{-\frac{1}{4}+\delta})$, where $\delta$ depends on the tails of the return times. When the return time
has a  (stretched) exponential tail, $\delta$ can be arbitrarily small.
We point out that the convergence rate we obtain is essentially optimal under the methods used, as it is well known that one cannot get a better result than $O(n^{-\frac{1}{4}})$ by means of the Skorokhod embedding theorem; see Remark~\ref{wass} for some details.

To derive the convergence rate, we employ techniques developed for stationary systems, particularly the martingale approximation method and the martingale version of the Skorokhod embedding theorem. The key ingredient is a secondary martingale-coboundary decomposition for the random tower map, which provides a good control over the sums of squares of the approximating martingale. To the best of our knowledge, this approach was first proposed in \cite{KKM18} to apply to families of dynamical systems. Then Paviato \cite{P24} generalized it to suspension flows. Here, we adopt a different method, following the ideas in \cite{BW71} and \cite{T05}, to construct the secondary martingale-coboundary decomposition.
Moreover, using the secondary martingale-coboundary decomposition, we obtain the QASIP for random Young towers with an ergodic driving system as an additional product.

The remainder of this paper is organized as follows. In Section 2, we introduce the random Young tower and recall some of its properties. In Section 3, we state main results of this paper.
In Section 4, we introduce the primary martingale-coboundary decomposition and construct a secondary martingale-coboundary decomposition.
In Section 5, we provide several technical lemmas. Based on these preparations, we proof the main results in Section 6.  In the last section, we present some examples to illustrate our results.

Throughout the paper, \\
(1) $1_A$ denotes the indicator function of measurable set $A$. \\
(2) We write $C_a$ to denote constants depending on $a$ and it may change from line to line.  \\
(3) $a_n=O_a(b_n)$ and $a_n\ll_a b_n$ mean that there exists a constant $C_a>0$ such that $a_n\le C_a b_n$ for all $n\ge 1$.\\
(4) $\E_{\mu_\omega}$ means the expectation with respect to $\mu_\omega$; $\E$ means the expectation of $\P$; $E$ means the expectation on an abstract probability space.\\
(5) We use $\rightarrow_{w}$ to denote the weak convergence in the sense of probability measures. \\
(6) $C[0,1]$ is the space of all continuous functions on $[0,1]$ equipped with the supremum distance $d$, that is
\[
d(x,y):=\sup_{t\in [0,1]}|x(t)-y(t)|, \quad x,y\in C[0,1].
\]
(7) $\P_X$ denotes the law/distribution of random variable $X$ and $X=_d Y$ means $X, Y$ sharing the same distribution.\\
(8) We use the notation $\mathcal{W}_p(X,Y)$ to mean $\mathcal{W}_p(\P_X, \P_Y)$ for the sake of simplicity.\\

\section{The setup}

\subsection{Random Young towers}
Let $(\Omega, \cal F, \P)$ be a probability space and $\sigma:\Omega\to \Omega$ an invertible ergodic measure-preserving transformation. Let $(X,\cal B)$ be a measurable space and $\Lambda\subset X$ a measurable set with finite measure $m$. Let $f_\omega:X\to X$ be a family of non-singular transformations. We say that $f_\omega$ admits a random Young tower if for almost every (a.e.) $\omega\in \Omega$ there exists a countable partition $\{\Lambda_j(\omega)\}_j$ of $\Lambda$ and a measurable return-time function $R_\omega:\Lambda\to \N$ such that $R_\omega$ is constant on each
$\Lambda_j(\omega)$ and $f_{\omega}^{R_\omega}(x)=f_{\sigma^{R_\omega(x)-1}\omega}\circ \cdots\circ f_{\sigma\omega}\circ f_\omega(x)\in \Lambda$ for $\P$-a.e. $\omega\in \Omega$ and $m$-a.e. $x\in \Lambda$.

For a.e. $\omega\in \Omega$, define a random tower by
\[
\Delta_\omega:=\Big\{(x,l)\in \Lambda\times \Z_+\big| x\in \bigcup_j\Lambda_j(\sigma^{-l}\omega), l\in\Z_+, 0\le l\le R_{\sigma^{-l}\omega}(x)-1   \Big\}.
\]
Denote by $\Delta_{\omega,l}:=\{(x,l)\in \Delta_\omega\}$ the $l$th level of the tower for $l\ge 1$, which is a copy of
$\{x\in \Lambda | R_{\sigma^{-l}\omega}(x)>l\}$ and $\Delta_{\omega,0}:=\Lambda$.
The random tower map $F_\omega: \Delta_{\omega}\to\Delta_{\sigma\omega}$ is given by
\begin{align*}
  F_\omega(x,l):=
  \begin{cases}
  (x,l+1), &  \hbox {if}\quad l+1<R_{\sigma^{-l}\omega}(x),\\
  (f_{\sigma^{-l}\omega}^{l+1}, 0),  &\hbox{if} \quad l+1=R_{\sigma^{-l}\omega}(x).
  \end{cases}
\end{align*}
The collection $\Delta:=\bigcup_{\omega\in \Omega}\{\omega\}\times \Delta_\omega$ is called a random Young tower. The fibred system
$\{F_\omega\}_{\omega\in \Omega}$ is called the random tower map. The skew product $F:\Delta\to \Delta$ is defined by
$F(\omega,x):=(\sigma\omega, F_\omega x)$.

Notice that $\{\Lambda_j(\sigma^{-l}\omega)\}_j$ induces a partition $\mathcal P_{\omega,l}$ on the $l$th level of the tower $\Delta_\omega$:
\[
\mathcal P_{\omega,l}:=\big\{\Lambda_j(\sigma^{-l}\omega)\big|R_{\omega}|_{\Lambda_j(\sigma^{-l}\omega)}\ge l+1, l\in\Z_+ \big\}.
\]
We also let $\mathcal P_{\omega}$ be the corresponding partitions of $\Delta_\omega$.

Next, we extend $R_\omega$ to $\Delta_\omega$ (still denoted by $R_\omega$) by setting
$R_\omega(x,l)=R_{\sigma^{-l}\omega}(x)-l$. Namely, $R_\omega$ denotes the first return time to the base of the tower $\Delta_{\sigma^{R_\omega}\omega}$.
Also, the reference measure $m$ and the $\sigma$-algebra on
$\Lambda$ naturally lift to $\Delta_\omega$. By abuse of notations we call the lifted measure $m$ and the lifted $\sigma$-algebra $\mathcal B_\omega$. Let $R_\omega^1(x)=R_\omega(x)$ and for $n\ge 2$, $R_\omega^n(x)=R_{\sigma^{R_\omega^{n-1}}\omega}(F_\omega^{n-1}(x))+R_\omega^{n-1}(x)$.
Then we define the separation time $s_\omega:\Delta_\omega\times\Delta_\omega\to \Z_+\cup\{\infty\}$ for a.e. $\omega\in \Omega$ by
\[
s_\omega(x,y):=\inf\{n\ge 0:F_\omega^{R_{\omega}^{n}}(x) \hbox{~and~} F_\omega^{R_{\omega}^{n}}(y) \hbox{~lie~in~distinct~elements~ of~} \mathcal P_{\sigma^{R_{\omega}^{n}}\omega}\}.
\]

We assume that the random Young tower satisfies the following properties.\\
(P1) {\bf Markov:} For each $\Lambda_j(\omega)$, the map $F_\omega^{R_\omega}|_{\Lambda_j(\omega)}:\Lambda_j(\omega)\to \Lambda$
is a bijection.\\
(P2) {\bf Bounded distortion:} There are constants $C>0$ and $\beta\in (0,1)$ such that for a.e. $\omega\in \Omega$ and each $\Lambda_j(\omega)$ the map $F_\omega^{R_\omega}|_{\Lambda_j(\omega)}$ and its inverse are non-singular w.r.t. $m$, and the corresponding Jacobian $JF_\omega^{R_\omega}|_{\Lambda_j(\omega)}$ is positive and for each $x,y\in \Lambda_j(\omega)$, we have
\[
\Big|\frac{JF_\omega^{R_\omega}(x)}{JF_\omega^{R_\omega}(y)}-1\Big|\le C\beta^{s_{\sigma^{R_\omega}\omega}(F_\omega^{R_\omega}(x,0), F_\omega^{R_\omega}(y,0))}.
\]
(P3) {\bf Weak expansion:} $\mathcal P_\omega$ is a generating partition for $F_\omega$, i.e. diameters of the partitions
$\bigvee_{j=0}^{n}F_{\omega}^{-j}\mathcal P_{\sigma^{j}\omega}$ tend to $0$ as $n\to \infty$.\\
(P4) {\bf Aperiodicity:} There are $N\in \N$, $\{\eps_i>0 | i=1,2,\ldots,N\}$ and $\{t_i\in \Z_+|i=1,2,\ldots,N\}$ with g.c.d.$\{t_i\}=1$ such that for a.e. $\omega\in \Omega$, all $1\le i\le N$,
\[
m\{x\in \Lambda|R_\omega(x)=t_i\}>\eps_i.
\]
(P5) {\bf Return-time asymptotics:} There are constants $C>0$, $a>1$, $b\ge 0$, $u>0$, $v>0$, a full measure subset $\Omega_1\subset\Omega$, and a random variable $n_1:\Omega_1\to \N$ such that
\begin{align*}
\begin{cases}
  m\{x\in \Lambda|R_\omega(x)>n\}\le C\frac{(\log n)^b}{n^a} \hbox{~whenever~} n\ge n_1(\omega),\\
  \P\{n_1(\omega)>n\}\le Ce^{-un^v}.
  \end{cases}
\end{align*}
(P6) {\bf Finiteness:} There exists a constant $M>0$ such that $m(\Delta_\omega)\le M$ for all $\omega\in \Omega$.\\
(P7) {\bf Annealed return-time asymptotics:} There are constants $C>0$, $\hat b\ge 0$, and $a>1$ such that
$\int_\Omega m\{x\in \Lambda|R_\omega(x)=n\}\rmd\P\le C\frac{(\log n)^{\hat b}}{n^{a+1}}$.

In addition, we also need to assume the quenched decay of correlation. We first introduce some function spaces. Below we let constants $u>0$, $v>0$, $a>1$, $b\ge 0$, $0<\beta<1$ be as in (P2) and (P4) above and set
\begin{align*}
\mathcal F_{\beta}^{+}:=\Big\{\phi:\Delta\to \R\big|&\exists C'_\phi>0, \forall I_\omega\in \mathcal P_\omega, \hbox{either~}
\phi_\omega|_{I_\omega}\equiv 0, \\
&\hbox{or~} \phi_\omega|_{I_\omega}>0 \hbox{~and~} \Big|\log\frac{\phi_\omega(x)}{\phi_\omega(y)}\Big|\le  C'_\phi
\beta^{s_\omega(x,y)}, \forall x,y\in I_\omega    \Big\},
\end{align*}
where $C'_\phi$ is called a Lipschitz constant for $\phi$.

Let $K:\Omega\to \R_+$ be a random variable with $\inf_{\omega\in\Omega} K_\omega>0$ and
\begin{align}\label{k}
\P\{\omega|K_\omega>n\}\le e^{-un^v}.
\end{align}
Define the space of bounded random Lipschitz functions as
\begin{align*}
\mathcal F_{\beta}^{K}:=\{&\phi:\Delta\to \R\big|\hbox{there~exist~constants~} M_\phi \hbox{~and~} C_\phi>0 \hbox{~such~that~for~a.e.~} \omega\in \Omega, \\
&\sup_{x\in \Delta_\omega}|\phi_\omega(x)|\le  M_\phi, \hbox{~and~} |\phi_\omega(x)-\phi_\omega(y)|\le  C_\phi K_\omega\beta^{s_\omega(x,y)}, \forall x,y\in \Delta_\omega\},
\end{align*}
where $C_\phi$ is also called a Lipschitz constant for $\phi$.

Throughout the paper, we define $\phi_\omega(\cdot):=\phi(\omega,\cdot)$ for any function $\phi:\bigcup_{\omega\in \Omega}\{\omega\}\times \Delta_\omega\to \R$.
\begin{prop}\label{suph}
For a.e. $\omega\in \Omega$, there is a family of $\{\mu_{\sigma^k \omega}\}$, which is equivariant  and absolutely continuous w.r.t. $m$, namely, $(F_{\sigma^k \omega})_*\mu_{\sigma^k \omega}=\mu_{\sigma^{k+1}\omega}$ with
$\mu_{\omega}=h_\omega m$. Moreover, $h_\omega\in \mathcal F_{\beta}^{+}\cap \mathcal F_{\beta}^{K_\omega}$ and
$\operatorname*{ess\,sup}_{\omega\in \Omega, x\in \Delta_\omega} h_\omega(x)<\infty$.
\end{prop}
\begin{proof}
See for example \cite[Theorem~4.1]{BBR19} and \cite[Theorem~2.1]{ABRV23}.
\end{proof}

(P8) {\bf Quenched decay of correlation:} For any small $\delta\in (0,a-1)$, there exists a full measure set $\Omega_0\subset \Omega$ and an integrable random variable $C_\omega$ on $\Omega_0$ s.t. for every $\phi\in L^{\infty}(\Delta, \mu)$, $\psi\in \mathcal F_{\beta}^{K}$, there is a constant
$C_{\phi,\psi}$ s.t. for every $\omega\in \Omega_0$,
\begin{align}\label{dec}
\Big|\int(\phi_{\sigma^n \omega}\circ F_{\omega}^{n})\psi_\omega \rmd m-\int\phi_{\sigma^n \omega}\rmd\mu_{\sigma^n \omega}\int \psi_\omega \rmd m  \Big|\le C_\omega C_{\phi,\psi}n^{-(a-1-\delta)}.
\end{align}

\begin{rem}\label{com}
We stress that (P8) is crucial for Proposition~\ref{ann-decay}. When the driving system is a Bernoulli scheme, by \cite[Theorem~4.2]{BBR19}, if the RYT satisfies the properties (P1)--(P7), then (P8) holds. In particular, if the return time of the RYT has an exponential tail, by \cite[Theorem~1.2.6]{Du15}, (P8) holds under the properties (P1)--(P5), without (P6)--(P7). When the driving system is an ergodic automorphism, under the properties (P1)--(P4), and a condition that $m\{x\in \Lambda|R_\omega(x)>n\}\le Ce^{-\theta n}$, Alves et al \cite{ABRV23} proved quenched exponential correlations decay for tower maps admitting exponential tails, but there is no regularity information on $C_\omega$.
\end{rem}

\subsection{Preliminaries}

\begin{de}[Random transfer operators]
For a.e. $\omega\in \Omega$, $P_\omega$ is called a random transfer operator for $F_\omega:\Delta_\omega\to \Delta_{\sigma\omega}$ if for any $\phi_{\sigma\omega}\in L^\infty(\Delta_{\sigma\omega}, \mu_{\sigma\omega})$, $\psi_\omega\in L^1(\Delta_\omega, \mu_\omega)$,
\begin{align}\label{tran}
\int \psi_\omega\phi_{\sigma\omega}\circ F_\omega \rmd\mu_\omega=\int P_\omega(\psi_\omega)\phi_{\sigma\omega}\rmd\mu_{
\sigma\omega}.
\end{align}
\end{de}
In addition, the following results hold for a.e. $\omega\in \Omega$.\\
(i) If $\psi\in L^p(\Delta, \mu)$ for $1\le p\le \infty$, then
\begin{align}\label{con}
\|P_\omega\psi_\omega\|_{L^p(\mu_{\sigma\omega})}\le \|\psi_\omega\|_{L^p(\mu_{\omega})}.
\end{align}
(ii) If $\psi\in L^1(\Delta, \mu)$, then
\begin{align}\label{cvm}
\E_{\mu_\omega}[\psi_{\sigma^i\omega}\circ F_\omega^i|(F_\omega^{i+1})^{-1}\mathcal B_{\sigma^{i+1}\omega}]
=[P_{\sigma^i\omega}(\psi_{\sigma^i\omega})]\circ F_\omega^{i+1} \quad\hbox{~in~} L^1(\mu_\omega).
\end{align}
(iii) If $\phi,\psi\in L^2(\Delta, \mu)$, then
\[
P_{\omega}^{i+k}(\psi_{\sigma^i\omega}\circ F_\omega^{i}\cdot \phi_\omega)=P_{\sigma^i\omega}^{k}(\psi_{\sigma^i\omega}\cdot P_\omega^i(\phi_\omega)) \quad\hbox{~in~} L^1(\mu_{\sigma^{i+k}\omega}),
\]
where $P_\omega^i:=P_{\sigma^{i-1}\omega}\circ\cdots\circ P_{\sigma\omega}\circ P_\omega$.

\begin{prop}\label{ann-decay}
Consider the RYT satisfying (P1)-(P8) with $a>1$. Let $\phi\in \mathcal F_{\beta}^{K}$. Then for any small $\delta\in (0,a-1)$, there is a constant $C_{h,\phi}>0$ such that
\[
\E\int\big|P_{\omega}^{n}(\phi_\omega-\int\phi_\omega \rmd\mu_{\omega})\big|\rmd\mu_{\sigma^n \omega}\le C_{h,\phi}n^{-(a-1-\delta)}.
\]
\end{prop}
\begin{proof}
For $\psi\in L^{\infty}(\Delta,\mu)$, by \eqref{tran}, we have
\begin{align*}
&\E\Big|\int P_{\omega}^{n}(\phi_\omega-\int\phi_\omega \rmd\mu_{\omega})\psi_{\sigma^n \omega}\rmd\mu_{\sigma^n \omega}\Big|
=\E\Big|\int (\phi_\omega-\int\phi_\omega \rmd\mu_{\omega})\psi_{\sigma^n \omega}\circ F_{\omega}^{n}\rmd\mu_{\omega}\Big|\\
&=\E\Big|\int (\phi_\omega-\int\phi_\omega\rmd\mu_{\omega})h_\omega\psi_{\sigma^n \omega}\circ F_{\omega}^{n}\rmd m\Big|.
\end{align*}
Then it follows from Propositions~\ref{suph} and \eqref{dec} that
\begin{align*}
&\E\Big|\int P_{\omega}^{n}(\phi_\omega-\int\phi_\omega \rmd\mu_{\omega})\psi_{\sigma^n \omega}\rmd\mu_{\sigma^n \omega}\Big|\\
\le &\operatorname*{ess\,sup}_{\omega\in \Omega, x\in \Delta_\omega} h_\omega(x)\E C_\omega C_{\phi,\psi}n^{-(a-1-\delta)}\\
=&C_{h,\phi,\psi}n^{-(a-1-\delta)}.
\end{align*}
Take $\psi_{\sigma^n \omega}=\hbox{sgn~} P_{\omega}^{n}(\phi_\omega-\int\phi_\omega \rmd\mu_{\omega})$, then we have
\[
\E\int\big|P_{\omega}^{n}(\phi_\omega-\int\phi_\omega \rmd\mu_{\omega})\big|\rmd\mu_{\sigma^n \omega}\le C_{h,\phi}n^{-(a-1-\delta)}.
\]
\end{proof}

\begin{prop}[Regularities]\label{reg}
Suppose that $\phi\in \mathcal F_{\beta}^{K}$ with a Lipschitz constant $C_\phi$. Define $\Phi_n(\omega,\cdot)=(P_\omega^n\phi_\omega)(\cdot)$ for any $n\in \N$. Then $\Phi_n\in \mathcal F_{\beta}^{K\circ \sigma^{-n}+C_{h,F}}$ with a Lipschitz constant $C_\phi$ and a constant $C_{h,F}$  depending on $h$ and $F$.
\end{prop}
\begin{proof}
See \cite[Lemma~4.4]{Su22} for details.
\end{proof}

\section{Statement of main results}
In this section, we state our main results. We first suppose that  $\phi\in \mathcal F_{\beta}^{K}$ and
$\int \phi_\omega \rmd\mu_\omega=0$. For a.e. $\omega\in \Omega$, denote the Birkhoff sum
$S_n^\omega:=\sum_{i=0}^{n-1}\phi_{\sigma^i\omega}\circ F_{\omega}^{i}$, and the variance
\[\Sigma_n^2(\omega):=\int \big(\sum_{i=0}^{n-1}\phi_{\sigma^i\omega}\circ F_{\omega}^{i}\big)^2 \rmd\mu_\omega.\]
We consider a continuous process $W_{n}^{\omega}(t)\in C[0,1]$ defined by
\[
W_{n}^{\omega}(t)=\frac{1}{\sqrt n}\Big[\sum_{i=0}^{[nt]-1}\phi_{\sigma^i\omega}\circ F_{\omega}^{i}+(nt-[nt])\phi_{\sigma^{[nt]}\omega}\circ F_{\omega}^{[nt]}\Big],\quad t\in[0,1].
\]

\begin{thm}\label{thnon1}
Consider the RYT satisfying (P1)-(P8) with $a>5$. Let $\delta>0$ be small s.t. $q=\frac{a-1-\delta}{\delta+1}\ge4$. Suppose that $\phi\in \mathcal F_{\beta}^{K}$ and
$\int \phi_\omega \rmd\mu_\omega=0$.  Then the following results hold. \\
$(1)$ There is a constant $\Sigma^2\ge 0$ such that
$\lim_{n\to \infty}\frac{1}{n}\Sigma_n^2(\omega)=\Sigma^2$ for a.e. $\omega\in \Omega$.\\
$(2)$ If $\Sigma^2=0$,  then $\phi$ is a coboundary, that is, there exists $\chi\in L^q(\Delta,\mu)$ such that for a.e. $\omega\in \Omega$,
\[\phi_{\omega}=\chi_{\sigma\omega}\circ F_\omega-\chi_\omega\quad \hbox{a.s.}\]
$(3)$ If $\Sigma^2>0$, then for a.e. $\omega\in \Omega$, $W_n^\omega\to_w W$ in $C[0,1]$, where $W$ is a Brownian motion with variance $\Sigma^2$.
\end{thm}

Furthermore, we aim to obtain the Wasserstein convergence rate in the quenched invariance principle for the random Young tower. Notice that we do not ask any regularity assumption on $\phi$ as a function of  $\omega,$  and any mixing assumption on the driving system $\sigma$. So we
cannot obtain a convergence rate on $\lim_{n\to \infty}\frac{1}{n}\Sigma_n^2(\omega)=\Sigma^2$. Hence we redefine a self-normalized continuous process.

For a.e. $\omega\in \Omega$ and every $t\in[0,1]$, set
\begin{align}\label{Nt}
N_n^\omega(t):=\min\{1\le k\le n: t\Sigma_n^2(\omega)\le\Sigma_{k}^2(\omega)\}.
\end{align}
Define a continuous process $\overline W_{n}^{\omega}(t)\in C[0,1]$ by
\begin{equation}\label{wnt}
\overline W_{n}^{\omega}(t):=\frac{1}{\Sigma_n(\omega)}\Big[\sum_{i=0}^{N_n^\omega-1}\phi_{\sigma^i\omega}\circ F_{\omega}^{i}+\frac{t\Sigma_n^2-\Sigma_{N_n^{\omega}-1}^2}{\Sigma_{N_n^{\omega}}^2-
\Sigma_{N_n^\omega-1}^2}\phi_{\sigma^{N_n^\omega}\omega}\circ F_{\omega}^{N_n^\omega}\Big],\quad t\in[0,1].
\end{equation}
%
\begin{thm}\label{thnon2}
Consider the RYT satisfying (P1)-(P8) with $a>5$. Let $\delta>0$ be small s.t. $q=\frac{a-1-\delta}{\delta+1}\ge4$. Suppose that $\phi\in \mathcal F_{\beta}^{K}$ and
$\int \phi_\omega \rmd\mu_\omega=0$. Then for a.e. $\omega\in \Omega$, there exists a constant $C=C_{\omega,q,\phi}>0$ such that
$\mathcal{W}_{\frac{q}{4}}(\overline W_{n}^{\omega},B)\leq C n^{-\frac{1}{4}+\frac{1}{2q}}$
for all $n\geq 1$, where $B$ is a standard Brownian motion.
\end{thm}

\begin{rem}\label{wass}
(1) Since $\mathcal{W}_{p}$ $\le \mathcal{W}_{q}$ for $p\le q$, Theorem~\ref{thnon2} provides an estimate for $\mathcal{W}_{p}(\overline W_{n}^\omega,B)$ for all $1\le p\le q/4$, $q\ge4$.

(2) The L\'{e}vy-Prokhorov distance $\pi(\mu,\nu)$, between two probability measures $\mu$ and $\nu$, is defined by
\begin{equation*}
\pi(\mu,\nu):=\inf\{\eps> 0: \mu(A)\le \nu(A^{\eps})+\eps \quad \hbox{for ~all ~closed ~sets~} A\in \mathcal{B}\}.
\end{equation*}
Here $A^{\eps}$ denotes the $\eps$-neighborhood of $A$. The quenched WIP in Theorem~\ref{thnon1}(3) is equivalent to $\lim_{n\to\infty}\pi(W_n^\omega, W)=0$ for a.e. $w\in \Omega$.

Our result also implies a convergence rate $\pi(\overline W_n^\omega, B)=O(n^{-\frac{q-2}{4(q+4)}})$ with respect to the L\'{e}vy-Prokhorov distance. Indeed, for two given probability measures $\mu$ and $\nu$, we have, for
$p\ge 1$, $\pi(\mu,\nu)\le \mathcal{W}_{p}(\mu,\nu)^{\frac{p}{p+1}}$.

(3)The exact convergence rate $r(q)=n^{-\frac{1}{4}+\frac{1}{2q}}$ can potentially be improved with more careful arguments. However, we emphasize that the main feature of the convergence rate we obtain, specifically $\inf r=n^{-\frac{1}{4}}$, is essentially optimal under the methods used. This result adopts the martingale version of Skorokhod embedding theorem (see Lemma~\ref{xnw}). As indicated  in \cite{B73,S72}, this method cannot yield a better result than $O(n^{-\frac{1}{4}})$.
\end{rem}

\section{Martingale-coboundary decompositions}
In this section, we introduce two types of martingale-coboundary decompositions for the RYT. In Section~4.1, we construct a reverse martingale difference sequence for the observable $\phi$. In Section~4.2, to control the sum of the square of the approximating martingale, we derive a secondary martingale-coboundary decomposition for the observable $\breve \phi$, which will be defined later.

\subsection{The primary martingale-coboundary decomposition}
\begin{lem}\label{pmd}
Consider the RYT satisfying (P1)-(P8) with $a>5$. Let $\delta>0$ be small s.t. $q=\frac{a-1-\delta}{\delta+1}\ge4$.
We suppose that $\phi\in \mathcal F_{\beta}^{K}$ and $\int\phi_\omega \rmd\mu_{\omega}=0$.
Let
\[
\chi_\omega:=\sum_{i\ge 1}P_{\sigma^{-i}\omega}^{i}(\phi_{\sigma^{-i}\omega}), \quad \chi(\omega, \cdot):=\chi_\omega(\cdot),
\]
\[
\psi_\omega:=\phi_{\omega}-\chi_{\sigma\omega}\circ F_\omega+\chi_\omega,\quad \psi(\omega,\cdot):=\psi_\omega(\cdot).
\]
Then $\chi,\psi\in L^q(\Delta,\mu)$, and for a.e. $\omega\in \Omega$,
\begin{align*}
\psi_\omega\in  {\rm ker} P_\omega, \quad
\sum_{i=0}^{n-1}\phi_{\sigma^i \omega}\circ F_\omega^i=\sum_{i=0}^{n-1}\psi_{\sigma^{i} \omega}\circ F_\omega^{i}+
\chi_{\sigma^n\omega}\circ F_\omega^n-\chi_\omega,
\end{align*}
namely, this is a martingale-coboundary decomposition where $\{\psi_{\sigma^{i} \omega}\circ F_\omega^{i}\}_{i\ge 0}$ is a reverse martingale difference sequence w.r.t. $\{(F_\omega^{i})^{-1}\mathcal B_{\sigma^{i} \omega}\}_{i\ge 0}$.
\end{lem}
\begin{proof}
Since $\phi\in L^\infty(\Delta,\mu)$, $\|P_\omega\phi_\omega\|_{L^\infty(\mu_{\sigma\omega})}\le \|\phi_\omega\|_{L^\infty(\mu_{\omega})}\le M_\phi$.
By Proposition~\ref{ann-decay}, we have
\begin{align*}
&\big\|\chi\big\|_{L^q(\Delta,\mu)}\le
\sum_{i\ge 1}\big\|P_{\sigma^{-i}\omega}^{i}(\phi_{\sigma^{-i}\omega})\big\|_{L^q(\Delta,\mu)}\\
&\le \sum_{i\ge 1}\Big(\E\int|P_{\sigma^{-i}\omega}^{i}(\phi_{\sigma^{-i}\omega})|\cdot |P_{\sigma^{-i}\omega}^{i}(\phi_{\sigma^{-i}\omega})|^{q-1}\rmd\mu_{\omega}\Big)^{1/q}\\
&\le \sum_{i\ge 1}\Big(\E\int|P_{\sigma^{-i}\omega}^{i}(\phi_{\sigma^{-i}\omega})|\rmd\mu_{\omega}\Big)^{1/q}
M_\phi^{\frac{q-1}{q}}\\
&\ll M_\phi^{\frac{q-1}{q}}\sum_{i\ge 1}i^{-\frac{a-1-\delta}{q}}<\infty.
\end{align*}
Then $\psi\in L^q(\Delta,\mu)$ follows from $\chi\in L^q(\Delta,\mu)$ and $\phi\in L^\infty(\Delta,\mu)$.

Moreover, for a.e. $\omega\in \Omega$,
\begin{align*}
P_\omega\psi_\omega&=P_\omega\phi_{\omega}-P_\omega(\chi_{\sigma\omega}\circ F_\omega)+P_\omega\chi_\omega
=P_\omega\phi_{\omega}-\chi_{\sigma\omega}+P_\omega\chi_\omega\\
&=P_\omega\phi_{\omega}-\sum_{i\ge 1}P_{\sigma^{-i}\sigma\omega}^{i}(\phi_{\sigma^{-i}\sigma\omega})+
\sum_{i\ge 1}P_{\sigma^{-i}\omega}^{i+1}(\phi_{\sigma^{-i}\omega})=0.
\end{align*}
By \eqref{cvm}, $\E_{\mu_\omega}[\psi_{\sigma^i\omega}\circ F_\omega^i|(F_\omega^{i+1})^{-1}\mathcal B_{\sigma^{i+1}\omega}]
=[P_{\sigma^i\omega}(\psi_{\sigma^i\omega})]\circ F_\omega^{i+1}=0$. So $\{\psi_{\sigma^{i} \omega}\circ F_\omega^{i}\}_{i\ge 0}$ is a reverse martingale difference sequence w.r.t. $\{(F_\omega^{i})^{-1}\mathcal B_{\sigma^{i} \omega}\}_{i\ge 0}$.
\end{proof}

\subsection{The secondary martingale-coboundary decomposition}

In this subsection, we first define $\breve{\phi}_\omega:=\big[P_\omega(\psi_{\omega}^{2}-\int\psi_{\omega}^{2}\rmd\mu_\omega)\big]\circ F_\omega$ and
$\breve{\phi}(\omega, \cdot):=\breve{\phi}_\omega(\cdot).$


To obtain the quenched invariance principle and estimate the corresponding convergence rate, we need to control the Birkhoff sums of $\breve{\phi}_\omega$. A martingale-coboundary decomposition for $\breve{\phi}_\omega$ is therefore highly useful; this is the focus of the current subsection.

\begin{lem}\label{smd}
Let $p\ge 2$. Suppose that $\breve{\phi}\in L^p(\Delta, \mu)$ with $\int \breve{\phi}_\omega \rmd\mu_\omega=0$, and that
\[
\Big\|\sum_{k=1}^{n}P_{\sigma^{-k}\omega}^{k}\breve{\phi}_{\sigma^{-k}\omega}\Big\|_{L^p(\Delta,\mu)}=O(1) \hbox{~as~} n\to\infty.
\]
Then there exist $\breve{\psi}$, $\breve{\chi}\in L^p(\Delta, \mu)$ s.t. for a.e. $\omega\in \Omega$,
\[
\breve{\psi}_\omega=\breve{\phi}_{\omega}-\breve{\chi}_{\sigma\omega}\circ F_\omega+\breve{\chi}_\omega,\quad \breve{\psi}_\omega\in {\rm ker} P_\omega.
\]
\end{lem}

\begin{proof}
We first claim that if
\[
\Big\|\sum_{k=1}^{n}P_{\sigma^{-k}\omega}^{k}\breve{\phi}_{\sigma^{-k}\omega}\Big\|_{L^p(\Delta,\mu)}=O(1) \hbox{~as~} n\to\infty,
\]
then there exists $\breve{\chi}\in L^p(\Delta, \mu)$ with $\breve{\chi}_\omega(\cdot)=\breve{\chi}(\omega, \cdot)$, s.t. for a.e. $\omega\in \Omega$,
\[
\breve{\phi}_{\omega}=\breve{\chi}_{\sigma\omega}\circ F_\omega-(P _\omega\breve{\chi}_{\omega})\circ F_\omega.
\]
Let $\breve{\psi}_{\omega}=\breve{\chi}_{\omega}-(P_\omega\breve{\chi}_{\omega})\circ F_\omega$ and
$\breve{\psi}(\omega, \cdot)=\breve{\psi}_\omega(\cdot)$. Then for a.e. $\omega\in \Omega$,
$P_\omega\breve{\psi}_\omega=0$ and
\begin{align*}
\breve{\phi}_{\omega}
&=\breve{\chi}_{\sigma\omega}\circ F_\omega-(P _\omega\breve{\chi}_{\omega})\circ F_\omega\\
&=\breve{\chi}_{\sigma\omega}\circ F_\omega+\breve{\psi}_{\omega}-\breve{\chi}_{\omega},
\end{align*}
i.e. $\breve{\psi}_\omega=\breve{\phi}_{\omega}-\breve{\chi}_{\sigma\omega}\circ F_\omega+\breve{\chi}_\omega$. Also, since $\breve{\chi}, \breve{\phi}\in L^p(\Delta, \mu)$, we have $\breve{\psi}\in L^p(\Delta, \mu)$. Hence the result holds.

Next, we aim to proof the claim, generalizing the ideas in~\cite{BW71}. In view of the identity
\[
\sum_{k=1}^{n}(n-k)P_{\sigma^{-k}\omega}^{k}\breve{\phi}_{\sigma^{-k}\omega}=
\sum_{j=1}^{n-1}\sum_{k=1}^{j}P_{\sigma^{-k}\omega}^{k}\breve{\phi}_{\sigma^{-k}\omega}
\]
and the present assumption, we have
\begin{align*}
&\Big\|\sum_{k=1}^{n}(1-\frac{k}{n})P_{\sigma^{-k}\omega}^{k}\breve{\phi}_{\sigma^{-k}\omega}\Big\|_{L^p(\Delta, \mu)}
=\Big\|\frac{1}{n}\sum_{j=1}^{n-1}\sum_{k=1}^{j}P_{\sigma^{-k}\omega}^{k}\breve{\phi}_{\sigma^{-k}\omega}\Big\|_{L^p(\Delta, \mu)}\\
&\le \frac{1}{n}\sum_{j=1}^{n-1}\Big\|\sum_{k=1}^{j}P_{\sigma^{-k}\omega}^{k}\breve{\phi}_{\sigma^{-k}\omega}\Big\|_{L^p(\Delta, \mu)}=O(1) \hbox{~as~} n\to \infty.
\end{align*}
Since in a reflexible Banach space, bounded sequences possess a weakly convergent sequence, it follows that there exists a sequence, still denoted by $n$, and
$\breve{\chi}\in L^p(\Delta, \mu)$ with $\breve{\chi}(\omega, \cdot)=\breve{\chi}_\omega(\cdot)$ s.t.
\[
\lim_{n\to\infty}\E\int f_\omega \cdot \sum_{k=1}^{n}(1-\frac{k}{n})P_{\sigma^{-k}\omega}^{k}\breve{\phi}_{\sigma^{-k}\omega}\rmd\mu_\omega=\E\int f_\omega \breve{\chi}_\omega \rmd\mu_\omega,
\]
for each $f\in L^q(\Delta, \mu)$ with $f(\omega, \cdot)=f_\omega(\cdot)$ and $\frac{1}{p}+\frac{1}{q}=1$.

Similarly,
\[
\lim_{n\to\infty}\E\int f_{\sigma\omega} \Big[\sum_{k=1}^{n}(1-\frac{k}{n})P_{\sigma^{-k}\sigma\omega}^{k}\breve{\phi}_{\sigma^{-k}\sigma\omega}-
\sum_{k=1}^{n}(1-\frac{k}{n})P_{\sigma^{-k}\omega}^{k+1}\breve{\phi}_{\sigma^{-k}\omega}\Big]\rmd\mu_{\sigma\omega}
=\E\int f_{\sigma\omega}(\breve{\chi}_{\sigma\omega}-P_\omega\breve{\chi}_{\omega} )\rmd\mu_{\sigma\omega}.
\]
Since
$\lim_{n\to\infty}\frac{1}{n}\big\|\sum_{k=1}^{n}P_{\sigma^{-k}\omega}^{k}\breve{\phi}_{\sigma^{-k}\omega}\big\|_{L^p(\Delta, \mu)}=0$, and
\begin{align*}
&\sum_{k=1}^{n}(1-\frac{k}{n})P_{\sigma^{-k}\sigma\omega}^{k}\breve{\phi}_{\sigma^{-k}\sigma\omega}-
\sum_{k=1}^{n}(1-\frac{k}{n})P_{\sigma^{-k}\omega}^{k+1}\breve{\phi}_{\sigma^{-k}\omega}\\
&=P_\omega\breve{\phi}_{\omega}-\frac{1}{n}\sum_{k=0}^{n-1}P_{\sigma^{-k}\omega}^{k+1}\breve{\phi}_{\sigma^{-k}\omega},
\end{align*}
we have
\begin{align*}
&\lim_{n\to\infty}\E\int f_{\sigma\omega} \Big[\sum_{k=1}^{n}(1-\frac{k}{n})P_{\sigma^{-k}\sigma\omega}^{k}\breve{\phi}_{\sigma^{-k}\sigma\omega}-
\sum_{k=1}^{n}(1-\frac{k}{n})P_{\sigma^{-k}\omega}^{k+1}\breve{\phi}_{\sigma^{-k}\omega}\Big]\rmd\mu_{\sigma\omega}\\
=&\lim_{n\to\infty}\E\int f_{\sigma\omega} \Big[P_\omega\breve{\phi}_{\omega}-\frac{1}{n}\sum_{k=0}^{n-1}P_{\sigma^{-k}\omega}^{k+1}\breve{\phi}_{\sigma^{-k}\omega}\Big]
\rmd\mu_{\sigma\omega}\\
=&\E\int f_{\sigma\omega} P_\omega\breve{\phi}_{\omega} \rmd\mu_{\sigma\omega}.
\end{align*}
Hence
\[
\E\int f_{\sigma\omega}(\breve{\chi}_{\sigma\omega}-P_\omega\breve{\chi}_{\omega} )\rmd\mu_{\sigma\omega}
=\E\int f_{\sigma\omega} P_\omega\breve{\phi}_{\omega}\rmd\mu_{\sigma\omega}
\]
for each $f\in L^q(\Delta, \mu)$, i.e.
\[
\E\int f_{\sigma\omega}\circ F_\omega(\breve{\chi}_{\sigma\omega}-P_\omega\breve{\chi}_{\omega} )\circ F_\omega \rmd\mu_{\omega}
=\E\int f_{\sigma\omega}\circ F_\omega \breve{\phi}_{\omega}\rmd\mu_{\omega}.
\]
So for a.e. $\omega\in \Omega$,
$\breve{\phi}_{\omega}=\breve{\chi}_{\sigma\omega}\circ F_\omega-(P_\omega\breve{\chi}_{\omega})\circ F_\omega$. The claim holds.
\end{proof}

In our setting, we consider the RYT satisfying (P1)-(P8) with $a>5$, and let $\delta>0$ be small s.t. $q=\frac{a-1-\delta}{\delta+1}\ge4$.
Suppose that $\phi\in \mathcal F_{\beta}^{K}$ and $\int \phi_\omega \rmd\mu_\omega=0$.
Then $\breve{\phi}\in L^{q/2}(\Delta, \mu)$ and $\int \breve{\phi}_\omega \rmd\mu_\omega=0$. If we have
\[
\Big\|\sum_{k=1}^{n}P_{\sigma^{-k}\omega}^{k}\breve{\phi}_{\sigma^{-k}\omega}\Big\|_{L^{q/2}(\Delta,\mu)}=O(1) \quad\hbox{~as~} n\to\infty,
\]
then by Lemma~\ref{smd}, there exist $\breve{\psi}$, $\breve{\chi}\in L^{q/2}(\Delta, \mu)$ s.t. for a.e. $\omega\in \Omega$,
\[\breve{\psi}_\omega=\breve{\phi}_{\omega}-\breve{\chi}_{\sigma\omega}\circ F_\omega+\breve{\chi}_\omega, \quad
\breve{\psi}_\omega\in {\rm ker} P_\omega.\]
We refer it as a {\em secondary} martingale-coboundary decomposition for $\breve{\phi}$.

\begin{lem}
\[
\Big\|\sum_{k=1}^{n}P_{\sigma^{-k}\omega}^{k}\breve{\phi}_{\sigma^{-k}\omega}\Big\|_{L^{q/2}(\Delta,\mu)}=O(1) \quad\hbox{~as~} n\to\infty.
\]
\end{lem}

\begin{proof}
Recall that $\breve{\phi}_\omega=\big[P_\omega(\psi_{\omega}^{2}-\int\psi_{\omega}^{2}\rmd\mu_\omega)\big]\circ F_\omega$, so
\[
\breve{\phi}_{\sigma^{-k}\omega}=\big[P_{\sigma^{-k}\omega}(\psi_{\sigma^{-k}\omega}^{2}
-\int\psi_{\sigma^{-k}\omega}^{2}\rmd\mu_{\sigma^{-k}\omega})\big]\circ F_{\sigma^{-k}\omega}.
\]
Hence
\begin{align*}
&\sum_{k=1}^{n}P_{\sigma^{-k}\omega}^{k}\breve{\phi}_{\sigma^{-k}\omega}\\
=&\sum_{k=1}^{n}P_{\sigma^{-k}\omega}^{k}\Big(\big[P_{\sigma^{-k}\omega}(\psi_{\sigma^{-k}\omega}^{2}
-\int\psi_{\sigma^{-k}\omega}^{2}\rmd\mu_{\sigma^{-k}\omega})\big]\circ F_{\sigma^{-k}\omega}\Big)\\
=&\sum_{k=1}^{n}P_{\sigma^{-k}\omega}^{k}(\psi_{\sigma^{-k}\omega}^{2}
-\int\psi_{\sigma^{-k}\omega}^{2}\rmd\mu_{\sigma^{-k}\omega}).
\end{align*}

Since $\psi_\omega=\phi_{\omega}-\chi_{\sigma\omega}\circ F_\omega+\chi_\omega$, we can write
\begin{align*}
&\psi_\omega^2=(\phi_{\omega}-\chi_{\sigma\omega}\circ F_\omega+\chi_\omega)^2\\
=&\phi_{\omega}^2+\chi_{\sigma\omega}^2\circ F_\omega+\chi_\omega^2
+2\phi_{\omega}\chi_\omega-2\phi_{\omega}\chi_{\sigma\omega}\circ F_\omega
-2\chi_{\sigma\omega}\circ F_\omega\chi_\omega\\
=&\phi_{\omega}^2-\chi_{\sigma\omega}^2\circ F_\omega+\chi_\omega^2
+2\phi_{\omega}\chi_\omega+2\chi_{\sigma\omega}^2\circ F_\omega-2\phi_{\omega}\chi_{\sigma\omega}\circ F_\omega
-2\chi_{\sigma\omega}\circ F_\omega\chi_\omega\\
=&\phi_{\omega}^2-\chi_{\sigma\omega}^2\circ F_\omega+\chi_\omega^2
+2\phi_{\omega}\chi_\omega-2\psi_\omega\chi_{\sigma\omega}\circ F_\omega.
\end{align*}
So
\begin{align*}
\psi_{\sigma^{-k}\omega}^{2}
=\phi_{\sigma^{-k}\omega}^2-\chi_{\sigma^{-k+1}\omega}^2\circ F_{\sigma^{-k}\omega}+\chi_{\sigma^{-k}\omega}^2
+2\phi_{\sigma^{-k}\omega}\chi_{\sigma^{-k}\omega}-2\psi_{\sigma^{-k}\omega}\chi_{\sigma^{-k+1}\omega}\circ F_{\sigma^{-k}\omega}.
\end{align*}
Notice that $P_{\sigma^{-k}\omega}(\psi_{\sigma^{-k}\omega}\chi_{\sigma^{-k+1}\omega}\circ F_{\sigma^{-k}\omega})=0$ and
$\int \psi_{\sigma^{-k}\omega}\chi_{\sigma^{-k+1}\omega}\circ F_{\sigma^{-k}\omega} \rmd\mu_{\sigma^{-k}\omega}=0$, so we have
\begin{align}
&\nonumber\big\|\sum_{k=1}^{n}P_{\sigma^{-k}\omega}^{k}\breve{\phi}_{\sigma^{-k}\omega}\big\|_{L^{q/2}(\Delta,\mu)}\\\label{P1}
\le &\sum_{k=1}^{n}\Big\|P_{\sigma^{-k}\omega}^{k}\big(\phi_{\sigma^{-k}\omega}^2-\int \phi_{\sigma^{-k}\omega}^2 \rmd\mu_{\sigma^{-k}\omega}\big)\Big\|_{L^{q/2}(\Delta,\mu)}\\\label{P2}
&\quad +2\sum_{k=1}^{n}\Big\|P_{\sigma^{-k}\omega}^{k}\big(\phi_{\sigma^{-k}\omega}\chi_{\sigma^{-k}\omega}-\int \phi_{\sigma^{-k}\omega}\chi_{\sigma^{-k}\omega} \rmd\mu_{\sigma^{-k}\omega}\big)\Big\|_{L^{q/2}(\Delta,\mu)}\\\label{P3}
&\quad\quad +\Big\|P_{\sigma^{-n}\omega}^{n}\big(\chi_{\sigma^{-n}\omega}^2-\int \chi_{\sigma^{-n}\omega}^2 \rmd\mu_{\sigma^{-n}\omega}\big)-\big(\chi_{\omega}^2-\int \chi_{\omega}^2\rmd\mu_{\omega}\big)\Big\|_{L^{q/2}(\Delta,\mu)}.
\end{align}

We now aim to estimate \eqref{P1}--\eqref{P3}. To eatimate \eqref{P1}, we first note that
$\sigma:\Omega\to\Omega$ is a stationary process, so it suffices to consider
\[
\sum_{k=1}^{n}\Big\|P_{\omega}^{k}\big(\phi_{\omega}^2-\int \phi_{\omega}^2 \rmd\mu_{\omega}\big)\Big\|_{L^{q/2}(\Delta,\mu)}.
\]
Since $\phi_{(\cdot)}^2-\E_{\mu_{(\cdot)}} \phi_{(\cdot)}^2 \in \mathcal F_{\beta}^{K}$, it follows from Proposition~\ref{ann-decay} that
\begin{align*}
&\Big(\E\int\big|P_{\omega}^{k}(\phi_\omega^2-\int\phi_\omega^2 \rmd\mu_{\omega})\big|^q\rmd\mu_{\sigma^k \omega}\Big)^{1/q}\\
&\le (2M_\phi^2)^{\frac{q-1}{q}}\Big(\E\int\big|P_{\omega}^{k}(\phi_\omega^2-\int\phi_\omega^2 \rmd\mu_{\omega})\big|\rmd\mu_{\sigma^k \omega}\Big)^{1/q}\\
&\ll k^{-\frac{a-1-\delta}{q}}=k^{-(1+\delta)}.
\end{align*}
Hence
\[
\sum_{k=1}^{n}\Big\|P_{\omega}^{k}\big(\phi_{\omega}^2-\int \phi_{\omega}^2 d\mu_{\sigma\omega}\big)\Big\|_{L^{q/2}(\Delta,\mu)}\ll \sum_{k=1}^{n}k^{-(1+\delta)}=O(1)
\hbox{~as~} n\to \infty.
\]

As for \eqref{P2}, it suffices to consider
\[
\sum_{k=1}^{n}\Big\|P_{\omega}^{k}\big(\phi_{\omega}\chi_{\omega}-\int \phi_{\omega}\chi_{\omega} \rmd\mu_{\omega}\big)\Big\|_{L^{q/2}(\Delta,\mu)}.
\]
We recall that $\chi_\omega=\sum_{i\ge 1}P_{\sigma^{-i}\omega}^{i}(\phi_{\sigma^{-i}\omega})$. Then we have
\begin{align*}
&\Big\|P_{\omega}^{k}\big[\phi_{\omega}\chi_\omega-\int \phi_{\omega}\chi_\omega \rmd\mu_\omega\big]\Big\|_{L^{q/2}(\Delta,\mu)}\\
\le &\sum_{i\ge 1}\Big\|P_{\omega}^{k}\big[\phi_{\omega} P_{\sigma^{-i}\omega}^{i}(\phi_{\sigma^{-i}\omega})-\int \phi_{\omega}P_{\sigma^{-i}\omega}^{i}(\phi_{\sigma^{-i}\omega}) \rmd\mu_\omega\big]\Big\|_{L^{q/2}(\Delta,\mu)}\\
\le &\sum_{i\le k}\Big\|P_{\omega}^{k}\big[\phi_{\omega} P_{\sigma^{-i}\omega}^{i}(\phi_{\sigma^{-i}\omega})-\int \phi_{\omega} P_{\sigma^{-i}\omega}^{i}(\phi_{\sigma^{-i}\omega}) \rmd\mu_{\omega}\big]\Big\|_{L^{q/2}(\Delta,\mu)}\\
&\quad\quad +\sum_{i>k}\Big\|P_{\omega}^{k}\big[\phi_{\omega} P_{\sigma^{-i}\omega}^{i}(\phi_{\sigma^{-i}\omega})-\int \phi_{\omega} P_{\sigma^{-i}\omega}^{i}(\phi_{\sigma^{-i}\omega}) \rmd\mu_{\omega}\big]\Big\|_{L^{q/2}(\Delta,\mu)}\\
=:&Y_1+Y_2.
\end{align*}
We note that $\phi_{\omega} P_{\sigma^{-i}\omega}^{i}(\phi_{\sigma^{-i}\omega})\in\mathcal F_{\beta}^{K_{\omega}+K_{\sigma^{-i}\omega}+C_{h,F}}$ with a Lipschitz constant $C_\phi M_\phi$. Indeed,
$\big\|\phi_{\omega} P_{\sigma^{-i}\omega}^{i}(\phi_{\sigma^{-i}\omega})\big\|_{L^\infty(\mu_{\omega})}\le M_\phi^2$ and by Proposition~\ref{reg}, for any $x,y\in \Delta_{\omega}$,
\begin{align*}
&|\phi_{\omega}(x) P_{\sigma^{-i}\omega}^{i}(\phi_{\sigma^{-i}\omega})(x)-\phi_{\omega}(y) P_{\sigma^{-i}\omega}^{i}(\phi_{\sigma^{-i}\omega})(y)|\\
\le &|\phi_{\omega}(x)-\phi_{\omega}(y)|M_\phi+|P_{\sigma^{-i}\omega}^{i}(\phi_{\sigma^{-i}\omega})(x)
-P_{\sigma^{-i}\omega}^{i}(\phi_{\sigma^{-i}\omega})(y)|M_\phi\\
\le &K_{\omega}\beta^{s_{\omega}(x,y)}C_\phi M_\phi
+(K_{\sigma^{-i}\omega}+C_{h,F})\beta^{s_{\omega}(x,y)}C_\phi M_\phi.
\end{align*}
Now, by Proposition~\ref{ann-decay}, we can estimate $Y_1$ that
\begin{align*}
&\sum_{i\le k}\Big\|P_{\omega}^{k}\big[\phi_{\omega} P_{\sigma^{-i}\omega}^{i}(\phi_{\sigma^{-i}\omega})-\int \phi_{\omega} P_{\sigma^{-i}\omega}^{i}(\phi_{\sigma^{-i}\omega}) \rmd\mu_{\omega}\big]\Big\|_{L^{q/2}(\Delta,\mu)}\\
\le &C_{\phi,q}\sum_{i\le k} \Big(\E\int\Big|P_{\omega}^{k}\big[\phi_{\omega} P_{\sigma^{-i}\omega}^{i}(\phi_{\sigma^{-i}\omega})-\int \phi_{\omega} P_{\sigma^{-i}\omega}^{i}(\phi_{\sigma^{-i}\omega}) \rmd\mu_{\omega}\big]\Big|\rmd\mu_{\sigma^k \omega}\Big)^{2/q}\\
\le& C_{\phi,q}\sum_{i\le k}k^{-\frac{2(a-1-\delta)}{q}}\ll k^{-(1+2\delta)}.
\end{align*}

To estimate $Y_2$, by the duality, we have
\begin{align*}
&\int\Big|P_{\omega}^{k}\big[\phi_{\omega} P_{\sigma^{-i}\omega}^{i}(\phi_{\sigma^{-i}\omega})\big]\Big|\rmd\mu_{\sigma^k \omega}\\
=&\sup_{\xi:\|\xi\|_\infty\le 1}\int\xi P_{\omega}^{k}\big[\phi_{\omega} P_{\sigma^{-i}\omega}^{i}(\phi_{\sigma^{-i}\omega})\big]\rmd\mu_{\sigma^k \omega}\\
=&\sup_{\xi:\|\xi\|_\infty\le 1}\int\xi\circ F_{\omega}^{k}\cdot\phi_{\omega} P_{\sigma^{-i}\omega}^{i}(\phi_{\sigma^{-i}\omega})\rmd\mu_{\omega}\\
\ll & \int\big|P_{\sigma^{-i}\omega}^{i}(\phi_{\sigma^{-i}\omega})\big|\rmd\mu_{\omega}.
\end{align*}
Then
\begin{align*}
Y_2&\ll  \sum_{i>k} \Big(\E\int\big| P_{\sigma^{-i}\omega}^{i}(\phi_{\sigma^{-i}\omega})\big|\rmd\mu_{ \omega}\Big)^{2/q}\\
&\ll \sum_{i>k} i^{-\frac{2(a-1-\delta)}{q}}\ll k^{-(1+2\delta)}.
\end{align*}
Based on the estimates on $Y_1$ and $Y_2$, we have
\[
\Big\|P_{\omega}^{k}\big[\phi_{\omega}\chi_\omega-\int \phi_{\omega}\chi_\omega \rmd\mu_\omega\big]\Big\|_{L^{q/2}(\Delta,\mu)}\ll k^{-(1+2\delta)}.
\]
Hence $\eqref{P2}\ll \sum_{k=1}^{n}k^{-(1+2\delta)}=O(1)$ as $n\to \infty$.

As for \eqref{P3}, by \eqref{con} and the fact that $\chi\in L^q(\Delta,\mu)$, we have
\begin{align*}
\eqref{P3}&\le \Big\|P_{\sigma^{-n}\omega}^{n}\big(\chi_{\sigma^{-n}\omega}^2-\int \chi_{\sigma^{-n}\omega}^2 \rmd\mu_{\sigma^{-n}\omega}\big)\Big\|_{L^{q/2}(\Delta,\mu)}+\Big\|\chi_{\omega}^2-\int \chi_{\omega}^2\rmd\mu_{\omega}\Big\|_{L^{q/2}(\Delta,\mu)}\\
&\le 4\|\chi\|_{L^q(\Delta,\mu)}.
\end{align*}
Based on the above estimates, the result follows.
\end{proof}

%
%
%

\section{Auxiliary lemmas}
\begin{lem}\label{est}
Suppose that $\phi\in L^p(\Delta, \mu)$ for $p\ge 2$. Then for a.e. $\omega\in \Omega$,
\[
\int \big|\phi_{\sigma^n \omega}\circ F_{\omega}^{n}\big|^p\rmd\mu_\omega=o(n), \int \big|\phi_{\sigma^n \omega}\circ F_{\omega}^{n}\big|^2\rmd\mu_\omega=O_{\omega,p}(n^{2/p}).
\]
\[
\phi_{\sigma^n \omega}\circ F_{\omega}^{n}(x)=O_{\omega,p}(n^{2/p}(\log n)^{1/p}(\log\log n)^{2/p}) \quad \text{a.s.} \ x\in \Delta_\omega.
\]
\end{lem}

\begin{proof}
By Birkhoff's ergodic theory, for a.e. $\omega\in \Omega$,
\[
\lim_{n\to\infty}\frac{1}{n}\sum_{k=0}^{n-1}\int \big|\phi_{\sigma^k \omega}\circ F_{\omega}^{k}\big|^p\rmd\mu_\omega
=\E\int|\phi_\omega|^p\rmd\mu_\omega<\infty.
\]
So $\int |\phi_{\sigma^n \omega}\circ F_{\omega}^{n}|^p\rmd\mu_\omega=o(n)$ and
\[
\int \big|\phi_{\sigma^n \omega}\circ F_{\omega}^{n}\big|^2\rmd\mu_\omega\le \big(\int \big|\phi_{\sigma^n \omega}\circ F_{\omega}^{n}\big|^p\rmd\mu_\omega\big)^{2/p}=O_{\omega,p}(n^{2/p}).
\]

Since
\begin{align*}
&\int \Big|\frac{\phi_{\sigma^n \omega}\circ F_{\omega}^{n}}{n^{2/p}(\log n)^{1/p}(\log\log n)^{2/p}}\Big|^p\rmd\mu_\omega\\
=&o(\frac{n}{n^2\log n(\log\log n)^2})=o(\frac{1}{n\log n(\log\log n)^2}),
\end{align*}
by the Borel-Cantelli lemma, we have $\phi_{\sigma^n \omega}\circ F_{\omega}^{n}(x)=O_{\omega,p}(n^{2/p}(\log n)^{1/p}(\log\log n)^{2/p})$ a.s.$ \ x\in \Delta_\omega$.
\end{proof}

\begin{lem}\label{me}
Suppose that $\phi\in L^p(\Delta, \mu)$ for $p\ge 2$. Then for a.e. $\omega\in \Omega$,
\[
\Big\|\max_{0\le k\le n-1}|\phi_{\sigma^k \omega}\circ F_{\omega}^{k}|\Big\|_{L^{p}(\mu_\omega)}=O_{\omega,p}(n^{1/p}).
\]
\end{lem}
\begin{proof}
By Birkhoff's ergodic theory, for a.e. $\omega\in \Omega$,
\[
\lim_{n\to\infty}\frac{1}{n}\sum_{k=0}^{n-1}\int \big|\phi_{\sigma^k \omega}\circ F_{\omega}^{k}\big|^p\rmd\mu_\omega
=\E\int|\phi_\omega|^p\rmd\mu_\omega<\infty.
\]
So $\sum_{k=0}^{n-1}\int |\phi_{\sigma^k \omega}\circ F_{\omega}^{k}|^p\rmd\mu_\omega=O_{\omega,p}(n)$.
Since
\[
\int \max_{0\le k\le n-1}\big|\phi_{\sigma^k \omega}\circ F_{\omega}^{k}\big|^p\rmd\mu_\omega
\le \sum_{k=0}^{n-1}\int \big|\phi_{\sigma^k \omega}\circ F_{\omega}^{k}\big|^p\rmd\mu_\omega=O_{\omega,p}(n),
\]
we have $\big\|\max_{0\le k\le n-1}|\phi_{\sigma^k \omega}\circ F_{\omega}^{k}|\big\|_{L^{p}(\mu_\omega)}=O_{\omega,p}(n^{1/p})$.
\end{proof}

\begin{lem}\label{mme}
Suppose that $\psi\in L^p(\Delta, \mu)$ for $p\ge 2$. Let $\{\psi_{\sigma^{n-i} \omega}\circ F_{\omega}^{n-i}\}_{1\le i\le n}$ be a sequence of martingale differences w.r.t the filtration
$\{(F_{\omega}^{n-i})^{-1}\cal B_{\sigma^{n-i} \omega}\}_{1\le i\le n}$ for a fixed $n\ge 1$. Then for a.e. $\omega\in \Omega$,
\[
\Big\|\max_{1\le k\le n}\big|\sum_{i=1}^{k}\psi_{\sigma^{n-i} \omega}\circ F_{\omega}^{n-i}\big|\Big\|_{L^{p}(\mu_\omega)}=O_{\omega,p}(n^{1/2}).
\]
\end{lem}
\begin{proof}
By the Burkholder-Davis-Gundy inequality and Minkowski inequality, there is a constant $C_p$ such that for a.e. $\omega\in \Omega$,
\begin{align*}
&\Big\|\max_{1\le k\le n}\big|\sum_{i=1}^{k}\psi_{\sigma^{n-i} \omega}\circ F_{\omega}^{n-i}\big|\Big\|_{L^{p}(\mu_\omega)}\\
\le & C_p\Big\|\Big(\sum_{i=1}^{n}\psi_{\sigma^{n-i} \omega}^2\circ F_{\omega}^{n-i}\Big)^{1/2}\Big\|_{L^{p}(\mu_\omega)}\\
\le &C_p\Big(\sum_{i=1}^{n}\big\|\psi_{\sigma^{n-i} \omega}^2\circ F_{\omega}^{n-i}\big\|_{L^{p/2}(\mu_\omega)}\Big)^{1/2}\\
=&O_{\omega,p}(n^{1/2}),
\end{align*}
where the last equality is due to $\E\|\psi_\omega^2\|_{L^{p/2}(\mu_\omega)}\le (\E\int|\psi_\omega|^p\rmd\mu_\omega)^{2/p}<\infty$ and Birkhoff's ergodic theorem.
\end{proof}

\begin{lem}\label{mom}
Consider the RYT satisfying (P1)-(P8) with $a>5$. Let $n\ge 1$ and  $\delta>0$ be small s.t. $q=\frac{a-1-\delta}{\delta+1}\ge4$. We suppose that
$\phi\in \mathcal F_{\beta}^{K}$and $\int\phi_\omega \rmd\mu_{\omega}=0$.
Then for a.e. $\omega\in \Omega$,
\[
\Big\|\max_{1\le k\le n}\big|\sum_{i=0}^{k-1}\phi_{\sigma^i \omega}\circ F_{\omega}^{i}\big|\Big\|_{L^{q}(\mu_\omega)}=
O_{\omega,q}(n^{1/2}).
\]
\end{lem}
\begin{proof}
By Lemma~\ref{pmd}, we can write
\[
\sum_{i=1}^{k}\phi_{\sigma^{n-i} \omega}\circ F_{\omega}^{n-i}=\sum_{i=1}^{k}\psi_{\sigma^{n-i} \omega}\circ F_{\omega}^{n-i}+\chi_{\sigma^n\omega}\circ F_\omega^n-\chi_{\sigma^{n-k}\omega}\circ F_\omega^{n-k}.
\]
Then it follows from Lemmas~\ref{me} and \ref{mme} that
\begin{align*}
&\Big\|\max_{1\le k\le n}\big|\sum_{i=1}^{k}\phi_{\sigma^{n-i} \omega}\circ F_{\omega}^{n-i}\big|\Big\|_{L^{q}(\mu_\omega)}\\
\le &\Big\|\max_{1\le k\le n}\big|\sum_{i=1}^{k}\psi_{\sigma^{n-i} \omega}\circ F_{\omega}^{n-i}\big|\Big\|_{L^{q}(\mu_\omega)}+\Big\|\max_{1\le k\le n}\big|\chi_{\sigma^n\omega}\circ F_\omega^n-\chi_{\sigma^{n-k}\omega}\circ F_\omega^{n-k}\big|\Big\|_{L^{q}(\mu_\omega)}\\
=&O_{\omega,q}(n^{1/2})+O_{\omega,q}(n^{1/q})
= O_{\omega,q}(n^{1/2}).
\end{align*}
Writing $\sum_{i=0}^{k-1}\phi_{\sigma^i \omega}\circ F_{\omega}^{i}=\sum_{i=1}^{n}\phi_{\sigma^{n-i} \omega}\circ F_{\omega}^{n-i}-\sum_{j=1}^{n-k}\phi_{\sigma^{n-j} \omega}\circ F_{\omega}^{n-j}$, we obtain
\[\Big\|\max_{1\le k\leq n}\big|\sum_{i=0}^{k-1}\phi_{\sigma^i \omega}\circ F_{\omega}^{i}\big|\Big\|_{L^{q}(\mu_\omega)}= O_{\omega,q}(n^{1/2}).\]
\end{proof}

\begin{cor}\label{smb}
Consider the RYT satisfying (P1)-(P8) with $a>5$. Let $n\ge 1$ and  $\delta>0$ be small s.t. $q=\frac{a-1-\delta}{\delta+1}\ge4$. We suppose that  $\phi\in \mathcal F_{\beta}^{K}$and $\int\phi_\omega \rmd\mu_{\omega}=0$. Then for a.e. $\omega\in \Omega$,
\[
\Big\|\max_{1\le k\le n}\sum_{i=0}^{k-1}\breve{\phi}_{\sigma^i \omega}\circ F_{\omega}^{i}\Big\|_{L^{q/2}(\mu_\omega)}=
 O_{\omega,q}(n^{1/2}).
\]
\end{cor}
\begin{proof}
This follows from the martingale-coboundary decomposition of $\breve \phi$ and the argument for Lemma~\ref{mom}.
\end{proof}

\begin{cor}(QASIP)
Consider the RYT satisfying (P1)-(P8) with $a>9$. Let $n\ge 1$ and  $\delta>0$ be small s.t. $q=\frac{a-1-\delta}{\delta+1}>8$. We suppose that  $\phi\in \mathcal F_{\beta}^{K}$and $\int\phi_\omega \rmd\mu_{\omega}=0$. Then for a.e. $\omega\in \Omega$, there exists a probability space supporting a sequence $\{Z_n\}$ of independent centered Gaussian random variables such that
\[
\sup_{1\le k\le n}\big|\sum_{i=1}^{k}\phi_{\sigma^i\omega}\circ F_\omega^i-\sum_{i=1}^{k}Z_i\big|=O(n^{1/4}(\log n)^{1/2}(\log\log n)^{1/4}) \quad
\text{a.s.~} x\in \Delta_\omega.
\]
\end{cor}
\begin{proof}
By the martingale approximation, for a.e. $\omega\in \Omega$,
\begin{align*}
\sum_{i=0}^{n-1}\phi_{\sigma^i \omega}\circ F_\omega^i=\sum_{i=0}^{n-1}\psi_{\sigma^{i} \omega}\circ F_\omega^{i}+
\chi_{\sigma^n\omega}\circ F_\omega^n-\chi_\omega,
\end{align*}
where $\chi,\psi\in L^q(\Delta,\mu)$, $q> 8$ and $\{\psi_{\sigma^{i} \omega}\circ F_\omega^{i}\}_{i\ge 0}$ is a reverse martingale difference sequence w.r.t. $\{(F_\omega^{i})^{-1}\mathcal B_{\sigma^{i} \omega}\}_{i\ge 0}$.

Since $\chi \in L^q(\Delta,\mu)$, by Lemma~\ref{est},
\[
\max_{0\le k\le n-1}|\chi_{\sigma^k\omega}\circ F_\omega^k|=O(n^{2/q}(\log n)^{1/q}(\log\log n)^{2/q}) \quad \text{a.s.} \ x\in \Delta_\omega.
\]
It remains to prove the QASIP for the sequence $\{\sum_{i=0}^{n-1}\psi_{\sigma^{i} \omega}\circ F_\omega^{i}\}$.

Recall that
\[
\breve{\phi}_\omega=\big[P_\omega(\psi_{\omega}^{2}-\int\psi_{\omega}^{2}\rmd\mu_\omega)\big]\circ F_\omega,\quad
\breve{\phi}(\omega, \cdot)=\breve{\phi}_\omega(\cdot).
\]
By the secondary martingale-coboundary decomposition, for a.e. $\omega\in \Omega$,
\[\breve{\phi}_{\omega}=\breve{\psi}_\omega+\breve{\chi}_{\sigma\omega}\circ F_\omega-\breve{\chi}_\omega\]
with $\breve{\psi}$, $\breve{\chi}\in L^{q/2}(\Delta, \mu)$.
 Hence,
\begin{align*}
&\sum_{i=0}^{n-1}\mathbb E_{\mu_\omega}(\psi_{\sigma^{i} \omega}^2\circ F_\omega^{i}|(\cal F_{\omega}^{i+1})^{-1}\cal B_{\sigma^{i+1} \omega})-\sum_{i=0}^{n-1}\mathbb E_{\mu_\omega}(\psi_{\sigma^{i} \omega}^2\circ F_\omega^{i})\\
=&\sum_{i=0}^{n-1}\Big[P_{\sigma^{i}\omega}\big(\psi_{\sigma^{i}\omega}^{2}-\int\psi_{\sigma^{i}\omega}^{2}
\rmd\mu_{\sigma^{i}\omega}\big)\Big]\circ F_{\omega}^{i+1}\\
=&\sum_{i=0}^{n-1}\breve{\phi}_{\sigma^{i}\omega}\circ F_{\omega}^{i}
=\sum_{i=0}^{n-1}\breve{\psi}_{\sigma^{i}\omega}\circ F_{\omega}^{i}+\breve{\chi}_{\sigma^n\omega}\circ F_{\sigma^n\omega}-\breve{\chi}_\omega.
\end{align*}
Since $\{\breve{\psi}_{\sigma^{i}\omega}\circ F_{\omega}^{i}\}$ is a sequence of $L^2$ reverse martingale differences,
by~\cite[Corollary~2.5]{CM15}, the QASIP holds with error rate $o((n\log\log n)^{1/2})$. This implies that the law of the iterated logarithm
holds, i.e. $\sum_{i=0}^{n-1}\breve{\psi}_{\sigma^{i}\omega}\circ F_{\omega}^{i}=O((n\log\log n)^{1/2})$ a.s.
Also, $\breve{\chi}_{\sigma^n\omega}\circ F_{\sigma^n\omega}-\breve{\chi}_\omega=O(n^{4/q}(\log n)^{2/q}(\log\log n)^{4/q})$ a.s.$\ x\in \Delta_\omega$.
Hence
\begin{align*}
&\sum_{i=0}^{n-1}\mathbb E_{\mu_\omega}(\psi_{\sigma^{i} \omega}^2\circ F_\omega^{i}|(\cal F_{\omega}^{i+1})^{-1}\cal B_{\sigma^{i+1} \omega})-\sum_{i=0}^{n-1}\mathbb E_{\mu_\omega}(\psi_{\sigma^{i} \omega}^2\circ F_\omega^{i})\\
&=O((n\log\log n)^{1/2})+O(n^{4/q}(\log n)^{2/q}(\log\log n)^{4/q})\\
&=O((n\log\log n)^{1/2}).
\end{align*}
By Corollary~2.8 in~\cite{CM15}, enlarging our probability space if necessary, it is possible to find a sequence $(Z_k)_{k\ge 1}$ of independent
centered Gaussion variables such that
\[
\sup_{1\le k\le n}\big|\sum_{i=1}^{k}\psi_{\sigma^i\omega}\circ F_\omega^i-\sum_{i=1}^{k}Z_i\big|=O(n^{1/4}(\log n)^{1/2}(\log\log n)^{1/4}) \quad
\text{a.s.~} x\in \Delta_\omega.
\]
Since
\[
\max_{0\le k\le n-1}|\chi_{\sigma^k\omega}\circ F_\omega^k|=O(n^{2/q}(\log n)^{1/q}(\log\log n)^{2/q}) \quad \text{a.s.} \ x\in \Delta_\omega,
\]
we obtain that
\begin{align*}
&\sup_{1\le k\le n}\big|\sum_{i=1}^{k}\phi_{\sigma^i\omega}\circ F_\omega^i-\sum_{i=1}^{k}Z_i\big|\\
=&O(n^{1/4}(\log n)^{1/2}(\log\log n)^{1/4})+ O(n^{2/q}(\log n)^{1/q}(\log\log n)^{2/q})\\
=&O(n^{1/4}(\log n)^{1/2}(\log\log n)^{1/4}) \quad\text{a.s.~} x\in \Delta_\omega.
\end{align*}
\end{proof}

\section{Proofs of main results}

\subsection{Proof of Theorem~\ref{thnon1}}
\setcounter{equation}{0}
\begin{lem}\label{var}
Consider the RYT satisfying (P1)-(P8) with $a>5$. Let $\delta>0$ be small s.t. $q=\frac{a-1-\delta}{\delta+1}\ge4$. Let $\eta_n^2(\omega):=\int \big(\sum_{i=0}^{n-1}\psi_{\sigma^i\omega}\circ F_{\omega}^{i}\big)^2 \rmd\mu_\omega$, $\Sigma^2:=\E\int \psi_\omega^2\rmd\mu_\omega$. Then for a.e. $\omega\in \Omega$,
\[
\Sigma_n^2(\omega)-\eta_n^2(\omega)=O_{\omega,\delta}(n^{\frac{1}{2}+\frac{1}{q}}),
\]
\[
\lim_{n\to \infty}\frac{\Sigma_n^2(\omega)}{n}=\lim_{n\to \infty}\frac{\eta_n^2(\omega)}{n}=\Sigma^2.
\]
\end{lem}
\begin{proof}
It follows from Lemma~\ref{pmd} and the H\"older inequality that
\begin{align*}
&\Sigma_n^2(\omega)-\eta_n^2(\omega)\\
=&\int\big(\sum_{i=0}^{n-1}\phi_{\sigma^i\omega}\circ F_{\omega}^{i}-\sum_{i=0}^{n-1}\psi_{\sigma^i\omega}\circ F_{\omega}^{i}\big)\big(\sum_{i=0}^{n-1}\phi_{\sigma^i\omega}\circ F_{\omega}^{i}+\sum_{i=0}^{n-1}\psi_{\sigma^i\omega}\circ F_{\omega}^{i}\big)\rmd\mu_\omega\\
=&\int\big(\chi_{\sigma^n\omega}\circ F_\omega^n-\chi_\omega\big)\big(2\sum_{i=0}^{n-1}\psi_{\sigma^i\omega}\circ F_{\omega}^{i}+\chi_{\sigma^n\omega}\circ F_\omega^n-\chi_\omega\big)\rmd\mu_\omega\\
=&\int2\big(\chi_{\sigma^n\omega}\circ F_\omega^n-\chi_\omega\big)\big(\sum_{i=0}^{n-1}\psi_{\sigma^i\omega}\circ F_{\omega}^{i}\big)\rmd\mu_\omega+\int\big(\chi_{\sigma^n\omega}\circ F_\omega^n-\chi_\omega\big)^2\rmd\mu_\omega\\
\le&2\big\|\chi_{\sigma^n\omega}\circ F_\omega^n-\chi_\omega\big\|_{L^2(\mu_\omega)}\big\|\sum_{i=0}^{n-1}\psi_{\sigma^i\omega}\circ F_{\omega}^{i}\big\|_{L^2(\mu_\omega)}+\int\big(\chi_{\sigma^n\omega}\circ F_\omega^n-\chi_\omega\big)^2\rmd\mu_\omega\\
\le& Cn^{\frac{1}{2}+\frac{1}{q}}+Cn^{\frac{2}{q}}=O_{\omega,\delta}(n^{\frac{1}{2}+\frac{1}{q}}),
\end{align*}
where the last inequality is due to Lemmas~\ref{est} and \ref{mme} with $\chi\in L^q(\Delta, \mu)$.

By Birkhoff's ergodic theory, for a.e. $\omega\in \Omega$,
\[
\lim_{n\to\infty}\frac{1}{n}\sum_{k=0}^{n-1}\int \psi_{\sigma^k \omega}^2\circ F_{\omega}^{k}\rmd\mu_\omega
=\E\int\psi_\omega^2\rmd\mu_\omega=\Sigma^2.
\]
Hence,
\[
\lim_{n\to \infty}\frac{\Sigma_n^2(\omega)}{n}=\lim_{n\to \infty}\frac{\eta_n^2(\omega)}{n}=\Sigma^2.
\]
\end{proof}

\begin{proof}[Proof of Theorem~\ref{thnon1}]
(1) The result follows from Lemma~\ref{var}.\\
(2) If $\Sigma^2=\E\int \psi_\omega^2\rmd\mu_\omega=\E\int (\phi_{\omega}-\chi_{\sigma\omega}\circ F_\omega+\chi_\omega)^2 \rmd\mu_\omega=0$, then for a.e. $\omega\in \Omega$,
\[
\phi_{\omega}-\chi_{\sigma\omega}\circ F_\omega+\chi_\omega=0\quad \mu_\omega-a.s.
\]
So $\phi$ is a coboundary with $\chi\in L^q(\Delta, \mu)$.\\
(3) Note that for a.e. $\omega\in \Omega$, $t\in [0,1]$,
\begin{align*}
&\frac{1}{n}\sum_{i=0}^{[nt]-1}[P_{\sigma^i\omega}(\psi_{\sigma^i\omega}^2)]\circ F_\omega^{i+1}\\
=&\frac{1}{n}\sum_{i=0}^{[nt]-1}\breve \phi_{\sigma^i\omega}\circ F_\omega^i+\frac{1}{n}\sum_{i=0}^{[nt]-1}\int \psi_{\sigma^i\omega}^2\circ F_\omega^i\rmd\mu_\omega.
\end{align*}
Since $\frac{1}{n}\|\sum_{i=0}^{[nt]-1}\breve \phi_{\sigma^i\omega}\circ F_\omega^i\|_{L^2(\mu_\omega)}=O_\omega(n^{-1/2})$ and $\frac{1}{n}\sum_{i=0}^{[nt]-1}\int \psi_{\sigma^i\omega}^2\circ F_\omega^i \rmd\mu_\omega\to t\Sigma^2$ a.s., we obtain that
\[
\frac{1}{n}\sum_{i=0}^{[nt]-1}[P_{\sigma^i\omega}(\psi_{\sigma^i\omega}^2)]\circ F_\omega^{i+1}\to_w t\Sigma^2.
\]
Then we can apply the Appendix A to the case
\[M_n^\omega(t)=\frac{1}{\sqrt n}\Big(\sum_{i=0}^{[nt]-1}\psi_{\sigma^i\omega}\circ F_\omega^i+(nt-[nt])\psi_{\sigma^{[nt]}\omega}\circ F_\omega^{[nt]}\Big)\]
to deduce that $M_n^\omega\to_wW$ in $C[0,1]$.

By the relation that $\phi_{\omega}=\psi_\omega+\chi_{\sigma\omega}\circ F_\omega-\chi_\omega$, we have
\[
\Big\|\sup_{t\in[0,1]}|W_{n}^\omega(t)-M_n^\omega(t)|\Big\|_{L^q(\mu_\omega)}\leq n^{-1/2}\Big\|\max_{1\le k\le n}\big|\chi_{\sigma^{k}\omega}\circ F_\omega^{k}-\chi_\omega\big|\Big\|_{L^q(\mu_\omega)}\to 0.
\]
Hence for a.e. $\omega\in \Omega$, $W_n^\omega\to_wW$ in $C[0,1]$.
\end{proof}

\subsection{Proof of Theorem~\ref{thnon2}}
\setcounter{equation}{0}
We recall that for a.e. $\omega\in \Omega$,
$\eta_n^2(\omega)=\int \big(\sum_{i=0}^{n-1}\psi_{\sigma^i\omega}\circ F_{\omega}^{i}\big)^2 \rmd\mu_\omega$. For every $t\in[0,1]$, set
\[
r_n^\omega(t):=\min\{1\le k\le n: t\eta_n^2(\omega)\le\eta_{k}^2(\omega)\}.
\]
Similar to $\overline W_n^\omega$, we define the following continuous processes $\overline M_{n}^\omega(t)\in C[0,1]$ by
\begin{equation}\label{exp}
\overline M_{n}^\omega(t):=\frac{1}{\eta_n(\omega)}\bigg[\sum_{i=0}^{r_n^\omega-1}\psi_{\sigma^i\omega}\circ  F_{\omega}^i+\frac{t\eta_n^2-\eta_{r_n^\omega-1}^2}{\eta_{r_n^\omega}^2-\eta_{r_n^\omega-1}^2}
\psi_{\sigma^{r_n^\omega}\omega}\circ  F_\omega^{r_n^\omega}\bigg],\quad t\in[0,1].
\end{equation}

{\bf Step 1. Estimation of the convergence rate between $\overline W_n^\omega$ and $\overline M_n^\omega$.}

\begin{lem}\label{wmn}
For a.e. $\omega\in \Omega$, there exists a constant $C=C_{\omega,q}>0$ such that for all $n\ge 1$ ,
\[\Big\|\sup_{t\in[0,1]}|\overline W_{n}^\omega(t)-\overline M_n^\omega(t)|\Big\|_{L^q(\mu_\omega)}\leq Cn^{-\frac{1}{4}+\frac{1}{2q}}.\]
\end{lem}
\begin{proof}
We can estimate that
\begin{align*}
&\Big\|\sup_{t\in[0,1]}\big|\frac{1}{\Sigma_n(\omega)}\sum_{i=0}^{N_n^\omega-1}\phi_{\sigma^i\omega}\circ  F_{\omega}^i-\frac{1}{\eta_n(\omega)}\sum_{i=0}^{r_n^\omega-1}
\psi_{\sigma^i\omega}\circ  F_{\omega}^i\big|\Big\|_{L^q(\mu_\omega)}\\
\le&\big|\frac{1}{\Sigma_n(\omega)}-\frac{1}{\eta_n(\omega)}\big|\Big\|\sup_{t\in[0,1]}\big|\sum_{i=0}^{N_n^\omega-1}
\phi_{\sigma^i\omega}\circ F_{\omega}^i\big|\Big\|_{L^q(\mu_\omega)}\\
&\quad\quad+\frac{1}{\eta_n(\omega)}\Big\|\sup_{t\in[0,1]}\big|\sum_{i=0}^{N_n^\omega-1}\phi_{\sigma^i\omega}\circ  F_{\omega}^i-\sum_{i=0}^{r_n^\omega-1}
\psi_{\sigma^i\omega}\circ  F_{\omega}^i\big|\Big\|_{L^q(\mu_\omega)}.
\end{align*}
For the first term, by Lemmas~\ref{mom} and ~\ref{var}, we have for a.e. $\omega\in \Omega$,
\begin{align*}
&\big|\frac{1}{\Sigma_n(\omega)}-\frac{1}{\eta_n(\omega)}\big|\Big\|\sup_{t\in[0,1]}\big|\sum_{i=0}^{N_n^\omega(t)-1}
\phi_{\sigma^i\omega}\circ  F_{\omega}^i\big|\Big\|_{L^q(\mu_\omega)}\\
= &\big|\frac{\Sigma_n(\omega)-\eta_n(\omega)}{\Sigma_n(\omega)\cdot \eta_n(\omega)}\big|\Big\|\max_{1\le k\le n}\big|\sum_{i=0}^{k-1}\phi_{\sigma^i\omega}\circ  F_{\omega}^i\big|\Big\|_{L^q(\mu_\omega)}\\
\le &Cn^{-1+\frac{1}{q}}\cdot n^\frac{1}{2}=Cn^{-\frac{1}{2}+\frac{1}{q}}.
\end{align*}
To deal with the second term, we first note that for a.e. $\omega\in \Omega$,
\begin{align*}
\max_{1\le j\le n}|\Sigma_j^2(\omega)-\eta_j^2(\omega)|\le Cn^{\frac{1}{2}+\frac{1}{q}},\quad\eta_{n+m}^{2}(\omega)-\eta_{n}^{2}(\omega)\ge C m.
\end{align*}
Then we have
\[
\eta_{r_n^\omega-1}^2+O(n^{\frac{1}{2}+\frac{1}{q}})< t\eta_{n}^{2}+O(n^{\frac{1}{2}+\frac{1}{q}})=t\Sigma_n^2\le\Sigma_{N_n^\omega}^2
=\eta_{N_n^\omega}^{2}+O(n^{\frac{1}{2}+\frac{1}{q}})<\eta_{N_n^\omega+m}^{2}-C m+O(n^{\frac{1}{2}+\frac{1}{q}}).
\]
Without loss of generality, we assume that $r_n^\omega>N_n^\omega$. Taking $m=r_n^\omega-N_n^\omega$, we have
\[
\sup_{t\in [0,1]}|r_n^\omega-N_n^\omega|\le\max_{1\le j\le n}|\eta_{j}^2(\omega)-\eta_{j-1}^2(\omega)|+ O(n^{\frac{1}{2}+\frac{1}{q}})\le Cn^{\frac{1}{2}+\frac{1}{q}}.
\]
Hence,
\begin{align*}
&\Big\|\sup_{t\in[0,1]}\big|(\sum_{i=0}^{N_n^\omega-1}-\sum_{i=0}^{r_n^\omega-1})\psi_{\sigma^i\omega}\circ  F_{\omega}^i\big|\Big\|_{L^q(\mu_\omega)}\\
=&\Big\|\sup_{t\in[0,1]}\big|\sum_{i=N_n^\omega}^{r_n^\omega-1}\psi_{\sigma^i\omega}\circ  F_{\omega}^i\big|\Big\|_{L^q(\mu_\omega)}\\
\le &\Big\|\sup_{t\in[0,1]}\big|\sum_{i=0}^{r_n^\omega-N_n^\omega}\psi_{\sigma^i\omega}\circ  F_{\omega}^i\big|\Big\|_{L^q(\mu_\omega)}\\
\le &{\sup_{t\in [0,1]}|r_n^\omega-N_n^\omega|}^{1/2}=O(n^{\frac{1}{4}+\frac{1}{2q}}).
\end{align*}
Since $\phi_{\omega}=\psi_\omega+\chi_{\sigma\omega}\circ F_\omega-\chi_\omega$ and $\chi,\psi\in L^q(\Delta,\mu)$, by Lemma~\ref{me}, we can estimate the second term that
\begin{align*}
&\frac{1}{\eta_n(\omega)}\Big\|\sup_{t\in[0,1]}\big|\sum_{i=0}^{N_n^\omega-1}\phi_{\sigma^i\omega}\circ  F_{\omega}^i-\sum_{i=0}^{r_n^\omega-1}
\psi_{\sigma^i\omega}\circ  F_{\omega}^i\big|\Big\|_{L^q(\mu_\omega)}\\
=&\frac{1}{\eta_n(\omega)}\Big\|\sup_{t\in[0,1]}\big|\sum_{i=0}^{N_n^\omega-1}
(\psi_{\sigma^i\omega}+\chi_{\sigma^{i+1}\omega}\circ F_{\sigma^{i}\omega}-\chi_{\sigma^i\omega})\circ F_{\omega}^i-\sum_{i=0}^{r_n^\omega-1}
\psi_{\sigma^i\omega}\circ  F_{\omega}^i\big|\Big\|_{L^q(\mu_\omega)}\\
\le &\frac{1}{\eta_n(\omega)}\Big\|\sup_{t\in[0,1]}\big|\sum_{i=N_n^\omega}^{r_n^\omega-1}\psi_{\sigma^i\omega}\circ  F_{\omega}^i\big|\Big\|_{L^q(\mu_\omega)}
+\frac{1}{\eta_n(\omega)}\Big\|\sup_{t\in[0,1]}\big|\chi_{\sigma^{N_n^\omega}\omega}\circ F_\omega^{N_n^\omega}-\chi_\omega\big|\Big\|_{L^q(\mu_\omega)}\\
\le &\frac{1}{\eta_n(\omega)}O(n^{\frac{1}{4}+\frac{1}{2q}})+\frac{2}{\eta_n(\omega)}\|\max_{0\le k\le n-1}|\chi_{\sigma^k\omega}\circ F_\omega^k|\|_{L^q(\mu_\omega)}\\
\le& Cn^{-\frac{1}{4}+\frac{1}{2q}}+Cn^{-\frac{1}{2}+\frac{1}{q}}\le Cn^{-\frac{1}{4}+\frac{1}{2q}}.
\end{align*}
Moreover, by Lemma~\ref{me}, we have
\begin{align*}
&\Big\|\sup_{t\in[0,1]}|\overline W_{n}^\omega(t)-\frac{1}{\Sigma_n(\omega)}\sum_{i=0}^{N_n^\omega-1}\phi_{\sigma^i\omega}\circ  F_{\omega}^i|\Big\|_{L^q(\mu_\omega)}\\
\le&\frac{1}{\Sigma_n(\omega)}\big\|\max_{0\le i\le n-1 }|\phi_{\sigma^i\omega}\circ  F_{\omega}^i|\big\|_{L^q(\mu_\omega)}\le Cn^{-\frac{1}{2}+\frac{1}{q}},
\end{align*}
and
\begin{align*}
&\Big\|\sup_{t\in[0,1]}|\overline M_{n}^\omega(t)-\frac{1}{\eta_n(\omega)}\sum_{i=0}^{r_n^\omega(t)-1}\psi_{\sigma^i\omega}\circ  F_{\omega}^i|\Big\|_{L^q(\mu_\omega)}\\
\le&\frac{1}{\eta_n(\omega)}\big\|\max_{0\le i\le n-1 }|\psi_{\sigma^i\omega}\circ  F_{\omega}^i|\big\|_{L^q(\mu_\omega)}\le Cn^{-\frac{1}{2}+\frac{1}{q}}.
\end{align*}
Then we obtain the conclusion based on the above estimates.
\end{proof}

Let $n\ge 1$. For a.e. $\omega\in \Omega$ and $1\le l\le n$, define
\[
V_{n,l}^{\omega}:=\sum_{i=1}^{l}\E_{\mu_\omega}\big(\frac{1}{\eta_n^2(\omega)}\psi_{\sigma^{n-i} \omega}^2\circ F_\omega^{n-i}\big|(F_{\omega}^{n-(i-1)})^{-1}\cal B_{\sigma^{n-(i-1)} \omega}\big).
\]
For the convenience, we set $V_{n,0}^{\omega}=0$.

For a.e. $\omega\in \Omega$, define a stochastic process $X_n^\omega$ with sample paths in $C[0,1]$ by
\begin{eqnarray}\label{xn}
X_{n}^\omega(t):=\frac{1}{\eta_n(\omega)}\Big[\sum_{i=1}^{k}\psi_{\sigma^{n-i} \omega}\circ F_\omega^{n-i}+\frac{tV_{n,n}^\omega-V_{n,k}^\omega}{V_{n,k+1}^\omega-V_{n,k}^\omega}\psi_{\sigma^{n-(k+1)} \omega}\circ F_\omega^{n-(k+1)}\Big],
\end{eqnarray}
if $V_{n,k}^\omega\leq tV_{n,n}^\omega<V_{n,k+1}^\omega$.
\vskip 3mm

{\bf Step 2. Estimation of the Wasserstein convergence rate between $X_n^\omega$ and $B$. }

\begin{prop}\label{vn1}
For a.e. $\omega\in \Omega$, there exists a constant $C=C_{\omega,\psi,q}>0$ such that for all $n\ge 1$ ,
\[\Big\|V_{n,n}^\omega-
1\Big\|_{L^{q/2}(\mu_\omega)}\le C n^{-\frac{1}{2}}.\]
\end{prop}

\begin{proof}
For a.e. $\omega\in \Omega$ and $1\le j\le n$, we define $\alpha_{j}^{2}(\omega)=\sum_{i=1}^{j}\int\psi_{\sigma^{n-i}\omega}^2\circ  F_\omega^{n-i}\rmd\mu_\omega$. Then $\alpha_{n}^{2}(\omega)=\eta_n^2(\omega)$ and
\[
\Big\|V_{n,n}^\omega-1\Big\|_{L^{q/2}(\mu_\omega)}=
\Big\|V_{n,n}^\omega-\frac{\alpha_{n}^{2}(\omega)}{\eta_n^2(\omega)}\Big\|_{L^{q/2}(\mu_\omega)}
\]
Hence,
\begin{align}\label{v2}
&\Big\|V_{n,n}^\omega-1\Big\|_{L^{q/2}(\mu_\omega)}\\\nonumber
=&\frac{1}{\eta_n^2(\omega)}\Big\|\sum_{i=1}^{n}\E_{\mu_\omega}(\psi_{\sigma^{n-i} \omega}^2\circ F_\omega^{n-i}|(F_{\omega}^{n-(i-1)})^{-1}\cal B_{\sigma^{n-(i-1)} \omega})-\sum_{i=1}^{n}\E_{\mu_\omega}(\psi_{\sigma^{n-i} \omega}^2\circ F_\omega^{n-i})\Big\|_{L^{q/2}(\mu_\omega)}\\\nonumber
=&\frac{1}{\eta_n^2(\omega)}\Big\|\sum_{i=1}^{n}\E_{\mu_\omega}\big(\psi_{\sigma^{n-i} \omega}^2\circ F_\omega^{n-i}-\E_{\mu_\omega}(\psi_{\sigma^{n-i} \omega}^2\circ F_\omega^{n-i})\big|( F_{\omega}^{n-(i-1)})^{-1}\cal B_{\sigma^{n-(i-1)} \omega}\big)\Big\|_{L^{q/2}(\mu_\omega)}\\\nonumber
=&\frac{1}{\eta_n^2(\omega)}\Big\|\sum_{i=1}^{n}\Big[P_{\sigma^{n-i}\omega}\big(\psi_{\sigma^{n-i}\omega}^{2}-\int\psi_{\sigma^{n-i}\omega}^{2}
\rmd\mu_{\sigma^{n-i}\omega}\big)\Big]\circ F_{\omega}^{n-(i-1)}
\Big\|_{L^{q/2}(\mu_\omega)}\\\nonumber
=&\frac{1}{\eta_n^2(\omega)}\Big\|\sum_{i=1}^{n}\breve{\phi}_{\sigma^{n-i}\omega}\circ F_{\omega}^{n-i}
\Big\|_{L^{q/2}(\mu_\omega)}\\\nonumber
\le&C_{\omega,\psi,q} n^{-\frac{1}{2}},
\end{align}
where the last inequality follows from Corollary~\ref{smb}.
\end{proof}

\begin{lem}\label{xnw}
For a.e. $\omega\in \Omega$ and for any $\varepsilon>0$, there exists a constant $C=C_{\omega,\varepsilon,\psi,q}>0$ such that $\mathcal{W}_{\frac{q}{2}}(X_{n}^\omega,B)\leq C n^{-(\frac{1}{4}-\varepsilon)}$ for all $n\ge 1$.
\end{lem}

\begin{proof}
The proofs are based on the ideas employed in the deterministic systems in \cite[Lemma~4.4]{LW24}. To obtain the convergence rate, we have to produce a bound of $\mathcal{W}_{\frac{q}{2}}(X_{n}^\omega, B)$ for a fixed $n\ge1$ and
a.e. $\omega\in \Omega$. It suffices to deal with a single row of the array
$\{\psi_{\sigma^{n-i} \omega}\circ F_\omega^{n-i},(F_{\omega}^{n-i})^{-1}\cal B_{\sigma^{n-i} \omega}\}$ for
$1\le i\le n$.

For $n\ge1$ and a.e. $\omega\in \Omega$, $\{\psi_{\sigma^{n-i} \omega}\circ F_{\omega}^{n-i}\}_{1\le i\le n}$ is
a sequence of martingale differences w.r.t. the filtration
$\{(F_{\omega}^{n-i})^{-1}\cal B_{\sigma^{n-i} \omega}\}_{1\le i\le n}$. By the Skorokhod embedding theorem (see \cite{HH80}), there exists a probability space (depending on $n,
\omega$) supporting a standard Brownian motion $B^\omega$, a sequence of nonnegative random variables $\tau_1^\omega,\ldots, \tau_n^\omega$ with $T_i^\omega=\sum_{j=1}^{i}\tau_j^\omega$, and a sequence of $\sigma$-field $\mathcal{F}_{i}^\omega$ generated by all events up to $T_i^\omega$ for $1\le i\le n$, such that \\
(1) for $1\le i\le n$, $\sum_{j=1}^{i}\frac{1}{\eta_n(\omega)}\psi_{\sigma^{n-j} \omega}\circ F_\omega^{n-j}=B^\omega(T_i^\omega)$;\\
(2)  for $1\le i\le n$,
\begin{align}\label{T2}
\E_{\mu_\omega}(\tau_{i}^\omega|\mathcal{F}_{i-1}^\omega)=\E_{\mu_\omega}\big(\frac{1}{\eta_n^2(\omega)}\psi_{\sigma^{n-i} \omega}^2\circ F_\omega^{n-i}\big|( F_{\omega}^{n-(i-1)})^{-1}\cal B_{\sigma^{n-(i-1)} \omega}\big);
\end{align}
(3) for any $q\ge2$, there exists a constant $C_{\omega, q}<\infty$ such that
\begin{align}\label{Tp}
\E_{\mu_\omega}((\tau_{i}^\omega)^{q/2}|\mathcal{F}_{i-1}^\omega)\leq C_{\omega,q}\E_{\mu_\omega}\Big(\big|\frac{1}{\eta_n(\omega)}\psi_{\sigma^{n-i} \omega}\circ F_\omega^{n-i}\big|^{q}\Big|(F_{\omega}^{n-(i-1)})^{-1}\cal B_{\sigma^{n-(i-1)} \omega}\Big).
\end{align}
Then on this probability space and for this Brownian motion, we aim to show that for any $\varepsilon>0$ there exists a constant $C>0$ such that
\begin{align*}
\Big\|\sup_{t\in[0,1]}|X_{n}^\omega(t)-B^\omega(t)|\Big\|_{L^{q/2}(\mu_\omega)}\leq C n^{-(\frac{1}{4}-\varepsilon)}.
\end{align*}
Thus the result follows from the definition of the Wasserstein distance.

By the Skorokhod embedding theorem, we can write \eqref{xn} as
\begin{align}\label{set}
X_{n}^\omega(t)=B^\omega(T_{k}^\omega)+\bigg(\frac{tV_{n,n}^\omega-V_{n,k}^\omega}{V_{n,k+1}^\omega-V_{n,k}^\omega}\bigg)
\big(B^\omega(T_{k+1}^\omega)-B^\omega(T_{k}^\omega)\big),\quad \hbox{if}~ V_{n,k}^\omega\leq tV_{n,n}^\omega<V_{n,k+1}^\omega.
\end{align}

1. We first estimate $|X_{n}^\omega-B^\omega|$ on the set $\{|T_{n}^\omega-1|> 1\}$.
 Note that \eqref{T2} implies
\[
T_j^\omega-V_{n,j}^{\omega}=\sum_{i=1}^{j}\big(\tau_{i}^{\omega}-\E_{\mu_\omega}(\tau_{i}^{\omega}|\mathcal{F}_{i-1}^\omega)\big),\quad 1\le j\le n.
\]
Therefore $\{T_j^\omega-V_{n,j}^{\omega}, \mathcal{F}_{j}^\omega, 1\le j\le n\}$ is a martingale. By the Burkholder inequality and Minkowski inequality, there is a constant $C_q>0$ such that for a.e. $\omega\in \Omega$,
\begin{align}\label{TV}
\nonumber&\Big\|\max_{1\le j\le n}|T_j^\omega-V_{n,j}^{\omega}|\Big\|_{L^{q/2}(\mu_\omega)} \\
\le \nonumber&C_q\Big\|\Big(\sum_{i=1}^{n}\big|\tau_{i}^{\omega}-\E_{\mu_\omega}(\tau_{i}^{\omega}|\mathcal{F}_{i-1}^\omega)
\big|^2\Big)^{1/2}\Big\|_{L^{q/2}(\mu_\omega)}\\
\le &C_q\Big(\sum_{i=1}^{n}\Big\|\big|\tau_{i}^{\omega}-\E_{\mu_\omega}(\tau_{i}^{\omega}|\mathcal{F}_{i-1}^\omega)\big|^2
\Big\|_{L^{q/4}(\mu_\omega)}\Big)^{1/2}.
\end{align}
By the conditional Jensen inequality and \eqref{Tp}, for $q\ge 4$, we have
\begin{align*}
&\Big\|\big|\tau_{i}^{\omega}-\E_{\mu_\omega}(\tau_{i}^{\omega}|\mathcal{F}_{i-1}^\omega)\big|^2\Big\|_{L^{q/4}(\mu_\omega)}
\ll \big\|\tau_i^\omega\big\|_{L^{q/2}(\mu_\omega)}^{2}\\
&\ll \big\|\frac{1}{\eta_n(\omega)}\psi_{\sigma^{n-i} \omega}\circ F_\omega^{n-i}\big\|_{L^{q}(\mu_\omega)}^{4}.
\end{align*}
Then we can continue to estimate \eqref{TV} that
\begin{align}\label{TVE}
\nonumber&\Big\|\max_{1\le j\le n}|T_j^\omega-V_{n,j}^{\omega}|\Big\|_{L^{q/2}(\mu_\omega)} \\
\ll\nonumber& \Big(\sum_{i=1}^{n}\big\|\frac{1}{\eta_n(\omega)}\psi_{\sigma^{n-i} \omega}\circ F_\omega^{n-i}\big\|_{L^{q}(\mu_\omega)}^{4}\Big)^{1/2}\\
=\nonumber&C_{\omega,q}\Big(\frac{1}{ \eta_n^4(\omega)}\sum_{i=1}^{n} \big\|\psi_{\sigma^{n-i} \omega}\circ F_\omega^{n-i}\big\|_{L^{q}(\mu_\omega)}^{4}\Big)^{1/2}\\
=&O_{\omega,q,\psi}(n^{-\frac{1}{2}}),
\end{align}
where the last equality is due to $\psi\in L^q(\Delta, \mu)$ and Birkhoff's ergodic theorem.

On the other hand, it follows from Proposition~\ref{vn1} that
\begin{align}\label{Vn1}
\|V_{n,n}^\omega-1\|_{L^{q/2}(\mu_\omega)}\le  C_{\omega,\psi,q} n^{-\frac{1}{2}}.
\end{align}
Based on the above estimates, by Chebyshev's inequality we have
\begin{equation}\label{eqq}
\begin{split}
&\mu_\omega(|T_{n}^\omega-1|>1)\le \E_{\mu_\omega}|T_{n}^\omega-1|^{q/2} \\
&\le 2^{q/2-1}\big(\E_{\mu_\omega}|T_{n}^\omega-V_{n,n}^\omega|^{q/2}+\E_{\mu_\omega}|V_{n,n}^\omega-1|^{q/2}\big)\le C_{\omega,\psi,q}n^{-\frac{q}{4}}.
\end{split}
\end{equation}
By Lemma~\ref{mme}, we know that $\Big\|\sup_{t\in[0,1]}|X_{n}^\omega(t)|\Big\|_{L^{q}(\mu_\omega)}<\infty$.
Hence, by H\"{o}lder's inequality and \eqref{eqq}, we deduce that
\begin{align*}
I:=&\Big\| 1_{\{|T_{n}^\omega-1|>1\}}\sup_{t\in[0,1]}|X_{n}^\omega(t)-B^\omega(t)| \Big\|_{L^{q/2}(\mu_\omega)}\\
\le &\big(\mu_\omega(|T_{n}^\omega-1|>1)\big)^{1/q}\Big\| \sup_{t\in[0,1]}|X_{n}^\omega(t)-B^\omega(t)|\Big\|_{L^{q}(\mu_\omega)}\\
\le &\big(\mu_\omega(|T_{n}^\omega-1|>1)\big)^{1/q}\Big(\Big\|\sup_{t\in[0,1]}|X_{n}^\omega(t)|\Big\|_{L^{q}(\mu_\omega)}
+\Big\|\sup_{t\in[0,1]}|B^\omega(t)|\Big\|_{L^{q}(\mu_\omega)}\bigg)\\
\le &C_{\omega,\psi,q} n^{-\frac{1}{4}}.
\end{align*}

2. We now estimate $|X_{n}^\omega-B^\omega|$ on the set $\{|T_{n}^\omega-1|\le 1\}$:
\begin{align*}
&\Big\| 1_{\{|T_{n}^\omega-1|\le 1\}}\sup_{t\in[0,1]}|X_{n}^\omega(t)-B^\omega(t)|\Big\|_{L^{q}(\mu_\omega)}\\
\le &\Big\| 1_{\{|T_{n}^\omega-1|\le 1\}}\sup_{t\in[0,1]}|X_{n}^\omega(t)-B^\omega(T_{k}^\omega)|\Big\|_{L^{q}(\mu_\omega)}+\Big\|1_{\{|T_{n}^\omega-1|\le 1\}}\sup_{t\in[0,1]}|B^\omega(T_{k}^\omega)-B^\omega(t)|\Big\|_{L^{q}(\mu_\omega)}\\
 =: & I_1 + I_2.
\end{align*}
For $I_1$, it follows from \eqref{set} that
\begin{align*}
&\sup_{t\in[0,1]}|X_{n}^\omega(t)-B^\omega(T_{k}^\omega)|\\
\le& \max_{0\le j\le n-1}|B^\omega(T_{j+1}^\omega)-B^\omega(T_{j}^\omega)|\\
=&\frac{1}{\eta_n(\omega)}\max_{0\le j\le n-1}|\psi_{\sigma^{n-(j+1)} \omega}\circ F_\omega^{n-(j+1)}|.
\end{align*}
By Lemma~\ref{me} and the fact that $\psi\in L^q(\Delta,\mu)$, we have
\begin{align*}\label{zeta}
I_1=&\Big\| 1_{\{|T_{n}^\omega-1|\le 1\}}\sup_{t\in[0,1]}|X_{n}^\omega(t)-B^\omega(T_{k}^\omega)|\Big\|_{L^{q}(\mu_\omega)}\\
\le &\frac{1}{\eta_n(\omega)}\Big\| \max_{0\le j\le n-1}|\psi_{\sigma^{n-(j+1)} \omega}\circ F_\omega^{n-(j+1)}|\Big\|_{L^{q}(\mu_\omega)}\\
\le& C_{\omega,\psi,q} n^{-\frac{1}{2}+\frac{1}{q}}.
\end{align*}

3. Finally, we consider $I_{2}$ on the set $\{|T_{n}^\omega-1|\le 1\}$. Take $q_1>q$, then it is well known that for a.e. $\omega\in \Omega$,
\begin{equation}\label{bts}
\mathbb{E}_{\mu_\omega}|B^\omega(t)-B^\omega(s)|^{q_1}\le c|t-s|^{\frac{q_1}{2}}, \quad \text{for~ all~} s,t\in [0,2].
\end{equation}
So it follows from Kolmogorov's continuity theorem that for each $0<\gamma<\frac{1}{2}-\frac{1}{2q_1}$, the process $B^\omega(\cdot)$ admits a version, still denoted by $B^\omega$, such that for almost all $x$, the sample path $t\mapsto B^\omega(t,x)$ is H\"{o}lder continuous with exponent $\gamma$ and
\begin{equation*}
\Big\|\sup_{s,t\in[0,2]\atop s\neq t}\frac{|B^\omega(s)-B^\omega(t)|}{|s-t|^{\gamma}}\Big\|_{L^{q_1}(\mu_\omega)}< \infty.
\end{equation*}
In particular,
\begin{equation}\label{holder}
\Big\|\sup_{s,t\in[0,2]\atop s\neq t}\frac{|B^\omega(s)-B^\omega(t)|}{|s-t|^{\gamma}}\Big\|_{L^{q}(\mu_\omega)}< \infty.
\end{equation}

As for $|T_{k}^\omega-t|$, by some calculations, we have
\begin{align*}
&\sup_{t\in[0,1]}|T_{k}^\omega-t|\le \max_{0\le k\le n-1}\sup_{t\in[\frac{V_{n,k}^\omega}{V_{n,n}^\omega},\frac{V_{n,k+1}^\omega}{V_{n,n}^\omega})}|T_{k}^\omega-t|\\
&\le \max_{0\le k\le n}\big|T_{k}^\omega-V_{n,k}^\omega\big| + 3\max_{0\le k\le n}\Big|V_{n,k}^\omega-\frac{V_{n,k}^\omega}{V_{n,n}^\omega}\Big| +\max_{0\le k\le n-1}\big|V_{n,k+1}^\omega-V_{n,k}^\omega\big|.
\end{align*}
Note that $T_{0}^\omega=V_{n,0}^\omega=0$ and $\gamma<\frac{1}{2}$, so
\[
\sup_{t\in[0,1]}|T_{k}^\omega-t|^{\gamma}\le \max_{1\le j\le n}\left|T_{j}^\omega-V_{n,j}^\omega\right|^{\gamma} + 3^\gamma\max_{1\le j\le n}\Big|V_{n,j}^\omega-\frac{V_{n,j}^\omega}{V_{n,n}^\omega}\Big|^{\gamma} +\max_{0\le j\le n-1}\left|V_{n,j+1}^\omega-V_{n,j}^\omega\right|^{\gamma}.
\]
Hence we have
\begin{align}
&\Big\|\sup_{t\in[0,1]}|T_{k}^\omega-t|^{\gamma}\Big\|_{L^{q}(\mu_\omega)}\nonumber\\
\le & \Big\|\max_{1\le j\le n}\big|T_{j}^\omega-V_{n,j}^\omega\big|\Big\|_{L^{\gamma q}(\mu_\omega)}^{\gamma} +3^\gamma \Big\|\max_{1\le j\le n}\big|V_{n,j}^\omega-\frac{V_{n,j}^\omega}{V_{n,n}^\omega}\big|\Big\|_ {L^{\gamma q}(\mu_\omega)}^{\gamma}
 +\Big\|\max_{0\le j\le n-1}\big|V_{n,j+1}^\omega-V_{n,j}^\omega\big|\Big\|_ {L^{\gamma q}(\mu_\omega)}^{\gamma}.
\end{align}
For the first term, since $\gamma< \frac{1}{2}$, it follows from \eqref{TVE} that
\begin{equation}\label{Tgam}
\Big\|\max_{1\le j\le n}|T_{j}^\omega-V_{n,j}^\omega|\Big\|_{L^{\gamma q}(\mu_\omega)}^{\gamma}\le C_{\omega,\psi,q}n^{-\frac{\gamma}{2}}.
\end{equation}
For the second term, since $|V_{n,j}^\omega-\frac{V_{n,j}^\omega}{V_{n,n}^\omega}|=V_{n,j}^\omega|1-\frac{1}{V_{n,n}^\omega}|$, we have
\[\max_{1\le j\le n}\Big|V_{n,j}^\omega-\frac{V_{n,j}^\omega}{V_{n,n}^\omega}\Big|=V_{n,n}^\omega\Big|1-\frac1{V_{n,n}^\omega}\Big|=|V_{n,n}^\omega-1|.\]
Hence by Proposition~\ref{vn1},
\begin{align}\label{Vgam}
\Big\|\max_{1\le j\le n}\big|V_{n,j}^\omega-\frac{V_{n,j}^\omega}{V_{n,n}^\omega}\big|\Big\|_ {L^{\gamma q}(\mu_\omega)}^{\gamma}
=\big\|V_{n,n}^\omega-1\big\|_{L^{\gamma q}(\mu_\omega)}^{\gamma}\le C_{\omega,\psi,q}n^{-\frac{\gamma}{2}}.
\end{align}
As for the last term, by \eqref{cvm}, for all $1\le j\le n$,
\begin{align*}
&|V_{n,j}^\omega-V_{n,j-1}^\omega|\\
=&\E_{\mu_\omega}\big(\frac{1}{\eta_n^2(\omega)}\psi_{\sigma^{n-j} \omega}^2\circ F_\omega^{n-j}\big|( F_{\omega}^{n-(j-1)})^{-1}\cal B_{\sigma^{n-(j-1)} \omega}\big)\\
=&\frac{1}{\eta_n^2(\omega)}\big(P_{\sigma^{n-j}\omega}\psi_{\sigma^{n-j}\omega}^{2}
\big)\circ F_{\omega}^{n-(j-1)}.
\end{align*}
Then we have
\begin{align}\label{mgam}
\nonumber&\Big\|\max_{0\le j\le n-1}\big|V_{n,j+1}^\omega-V_{n,j}^\omega\big|\Big\|_ {L^{\gamma q}(\mu_\omega)}^{\gamma}\\
=\nonumber&\frac{1}{(\eta_n^2(\omega))^\gamma}\Big\|\max_{1\le j\le n}\big|\big(P_{\sigma^{n-j}\omega}\psi_{\sigma^{n-j}\omega}^{2}
\big)\circ F_{\omega}^{n-(j-1)}\big|\Big\|_{L^{\gamma q}(\mu_\omega)}^{\gamma}\\
\le& C_{\omega,\psi,q}n^{-\left(\gamma-\frac{2\gamma}{q}\right)},
\end{align}
where the inequality follows from Lemma~\ref{me} and the fact that $P_{\sigma^{n-j}\omega}\psi_{\sigma^{n-j}\omega}^{2}\in L^{q/2}(\Delta,\mu)$.

Based on the above estimates \eqref{Tgam}--\eqref{mgam}, we have
\begin{align}\label{te}
\Big\|\sup_{t\in[0,1]}|T_{k}^\omega-t|^{\gamma}\Big\|_{L^{q}(\mu_\omega)}
\le C_{\omega,\psi,q}\Big(n^{-\frac{\gamma}{2}}+n^{-\frac{\gamma}{2}}+n^{-\big(\gamma-\frac{2\gamma}{q}\big)}\Big)
\le C_{\omega,\psi,q} n^{-\frac{\gamma}{2}},
\end{align}
where the last inequality holds since
$\gamma< \frac{1}{2}$, $1-\frac{2}{q}\ge\frac{1}{2}$.

On the set $\{|T_{n}^\omega-1|\le 1\}$, note that
\[
\sup_{t\in[0,1]}|B^\omega(T_{k}^\omega)-B^\omega(t)|\le \Big[\sup_{s,t\in[0,2]\atop s\neq t}\frac{|B^\omega(s)-B^\omega(t)|}{|s-t|^{\gamma}}\Big]\Big[\sup_{t\in[0,1]}|T_{k}^\omega-t|^{\gamma}\Big].
\]
Since $0<\gamma<\frac{1}{2}-\frac{1}{2q_1}$, by H\"{o}lder's inequality, \eqref{holder} and \eqref{te}, we have
\begin{align*}
I_2=&\Big\|1_{\{|T_{n}^\omega-1|\le 1\}}\sup_{t\in[0,1]}|B^\omega(T_{k}^\omega)-B^\omega(t)|\Big\|_{{L^{q/2}(\mu_\omega)}}\\
\le &\Big\|\Big[\sup_{s,t\in[0,2]\atop s\neq t}\frac{|B^\omega(s)-B^\omega(t)|}{|s-t|^{\gamma}}\Big]\Big[\sup_{t\in[0,1]}|T_{k}^\omega-t|^{\gamma}\Big]\Big\|_{{L^{q/2}(\mu_\omega)}}\\
\le &\Big\|\sup_{s,t\in[0,2]\atop s\neq t}\frac{|B^\omega(s)-B^\omega(t)|}{|s-t|^{\gamma}}\Big\|_{{L^{q}(\mu_\omega)}}\Big\|\sup_{t\in[0,1]}|T_{k}^\omega-t|^{\gamma}\Big\|_{{L^{q}(\mu_\omega)}}\\
\le & C_{\omega,\psi,q} n^{-\frac{\gamma}{2}}.
\end{align*}
Note that $q_1$ can be taken arbitrarily large in \eqref{bts}, which implies that $\gamma$ can be chosen sufficiently close to $\frac{1}{2}$. So for any $\varepsilon>0$, we can choose $q_1$ large enough such that $I_2\le C_{\omega,\psi,q,\varepsilon}n^{-\frac{1}{4}+\varepsilon}$. The result now follows from the above estimates for $I,I_1$ and $I_2$.
\end{proof}

{\bf Step 3. Estimation of the Wasserstein convergence rate between $\overline M_n^\omega$ and $X_n^\omega$.}

\vskip3mm

\begin{prop}\label{Zn}
For a.e. $\omega\in\Omega$ and $n\ge 1$, define
\[Z_n^\omega:=\max_{0\le i,l\le\sqrt{n}}\Big|\sum_{j=i[\sqrt{n}\;]}^{i[\sqrt{n}\;]+l-1}\psi_{\sigma^{n-j} \omega}\circ F_\omega^{n-j}\Big|.
\]
Then

$(a)$ $|\sum_{j=a}^{b-1}\psi_{\sigma^{n-j} \omega}\circ F_\omega^{n-j}|\le Z_n^\omega((b-a)(\sqrt n-1)^{-1}+3)$ for all $0\le a<b \le n$.

$(b)$ $\|Z_n^\omega\|_{L^{q}(\mu_\omega)}\le C_{\omega,\psi,q} n^{\frac{1}{4}+\frac{1}{2q}}$ for all $n\ge 1$.
\end{prop}

\begin{proof}
$(a)$ Choose $0\le l_1\le l_2\le \sqrt n$ such that $l_1$ is the least integer s.t. $l_1 [\sqrt{n}\;]> a$ and $l_2$ is the biggest integer s.t. $l_2 [\sqrt{n}\;] < b $. Then
$l_2-l_1\le \frac{b-a}{\sqrt n-1}+1$. Since
\[
\sum_{j=a}^{b-1}\psi_{\sigma^{n-j}\omega}\circ F_\omega^{n-j}=\Big(\sum_{j=a}^{l_1 [\sqrt{n}\;]-1}+\sum_{j=l_1 [\sqrt{n}\;]}^{l_2 [\sqrt{n}\;]-1}+\sum_{j=l_2 [\sqrt{n}\;]}^{b-1}\Big)
\psi_{\sigma^{n-j} \omega}\circ F_\omega^{n-j},
\]
we have $|\sum_{j=a}^{b-1}\psi_{\sigma^{n-j} \omega}\circ F_\omega^{n-j}|\le ((l_2-l_1)+2)Z_n^\omega
\le (\frac{b-a}{\sqrt{n}-1}+3) Z_n^\omega$ for a.e. $\omega\in\Omega$.

$(b)$ We have
\begin{align*}
\int |Z_n^\omega|^{q} \rmd\mu_\omega&\le \sum_{0\le i\le \sqrt n}\int \max_{0\le l\le \sqrt n}\big|\sum_{j=i[\sqrt{n}\;]}^{i[\sqrt{n}\;]+l-1}\psi_{\sigma^{n-j} \omega}\circ F_\omega^{n-j}\big|^{q}\rmd\mu_\omega\\
&= \sum_{0\le i\le \sqrt n}\int \max_{0\le l\le \sqrt n}\big|\sum_{j=0}^{l-1}\psi_{\sigma^{n-i[\sqrt{n}\;]-j} \omega}\circ F_\omega^{n-i[\sqrt{n}\;]-j}\big|^{q}\rmd\mu_\omega.
\end{align*}
Then it follows from Lemma~\ref{mme} that for a.e. $\omega\in\Omega$,
\[\|Z_n^\omega\|_{L^{q}(\mu_\omega)}\le C_{\omega,\psi,q} {\sqrt n}^{1/q}\cdot {\sqrt n}^{1/2}= C_{\omega,\psi,q} n^{\frac{1}{4}+\frac{1}{2q}}.\]

\end{proof}

Define a continuous transformation $g:C[0,1]\rightarrow C[0,1]$ by $g(u)(t):=u(1)-u(1-t)$.

\begin{lem}\label{wnxn}
For a.e. $\omega\in \Omega$, there exists a constant $C=C_{\omega,\psi,q}>0$ such that
\[\mathcal{W}_{\frac{q}{4}}(g\circ \overline M_{n}^\omega,  X_{n}^\omega)\leq Cn^{-\frac{1}{4}+\frac{1}{2q}} \hbox{~for~ all~} n\ge 1.\]
\end{lem}

\begin{proof}
For a.e. $\omega\in \Omega$ and $1\le j\le n$, we recall that $\alpha_{j}^{2}(\omega)=\sum_{i=1}^{j}\int\psi_{\sigma^{n-i}\omega}^2\circ  F_\omega^{n-i}\rm d\mu_\omega$. Then \[\alpha_{n}^{2}(\omega)=\eta_n^2(\omega), \quad \alpha_{p+q}^{2}(\omega)-\alpha_{p}^{2}(\omega)\ge C q,
\quad \forall p,q\in\N, 1\le p<p+q\le n.\]
We define
\[
\widehat M_n^\omega(t):=\frac{1}{\eta_n(\omega)}\Big[\sum_{i=1}^{l}\psi_{\sigma^{n-i}\omega}\circ  F_\omega^{n-i}+\frac{t\alpha_{n}^{2}-\alpha_{l}^{2}}{\alpha_{l+1}^{2}-\alpha_{l}^{2}}
\psi_{\sigma^{n-(l+1)}\omega}\circ  F_\omega^{n-(l+1)}\Big],
\]
if $\alpha_{l}^{2}(\omega)\leq t\alpha_{n}^{2}(\omega)<\alpha_{l+1}^{2}(\omega)$.

Recall that for $t\in[0,1]$ and $n\ge 1$, the random variable $k=k_{n,t}^w\in \N$ is defined by the equation
$V_{n,k}^\omega\leq tV_{n,n}^\omega<V_{n,k+1}^\omega$. For $t\in[0,1]$,
\begin{align*}
tV_{n,n}^\omega\cdot\alpha_{n}^{2}<V_{n,k+1}^\omega\cdot\alpha_{n}^{2}&=\alpha_{k+1}^{2}+V_{n,k+1}^\omega\cdot\alpha_{n}^{2}- \alpha_{k+1}^{2}\\
&
\le\alpha_{k+1+m}^{2}-C m+V_{n,k+1}^\omega\cdot\alpha_{n}^{2}- \alpha_{k+1}^{2}.
\end{align*}
On the other hand,
\[
tV_{n,n}^\omega\cdot\alpha_{n}^{2}=t\alpha_{n}^{2}+t(V_{n,n}^\omega\cdot\alpha_{n}^{2}- \alpha_{n}^{2})\ge \alpha_{l}^{2}+t(V_{n,n}^\omega\cdot\alpha_{n}^{2}- \alpha_{n}^{2}).
\]
i.e.
\[
\alpha_{l}^{2}+t(V_{n,n}^\omega\cdot\alpha_{n}^{2}- \alpha_{n}^{2})< \alpha_{k+1+m}^{2}-C m+V_{n,k+1}^\omega\cdot\alpha_{n}^{2}- \alpha_{k+1}^{2}.
\]
 Without loss of generality, we assume that $l\ge k$. Take $m=l-k$, then
 \[
 \alpha_{l}^{2}+t(V_{n,n}^\omega\cdot\alpha_{n}^{2}- \alpha_{n}^{2})< \alpha_{l+1}^{2}-C (l-k)+V_{n,k+1}^\omega\cdot\alpha_{n}^{2}- \alpha_{k+1}^{2}.
 \]
So
\[
|l-k|\le C(|\alpha_{l+1}^{2}-\alpha_{l}^{2}|+|V_{n,k+1}^\omega\cdot\alpha_{n}^{2}- \alpha_{k+1}^{2}|+|V_{n,n}^\omega\cdot\alpha_{n}^{2}- \alpha_{n}^{2}|).
\]
Since
\begin{align*}
&\Big\|\max_{1\le j\le n}\big|V_{n,j}^\omega\cdot \alpha_n^2(\omega)-\alpha_j^2(\omega)\big|\Big\|_{L^{q/2}(\mu_\omega)}\\
=&\Big\|\max_{1\le j\le n}\big|\sum_{i=1}^{j}\E_{\mu_\omega}(\psi_{\sigma^{n-i} \omega}^2\circ F_\omega^{n-i}|( F_{\omega}^{n-(i-1)})^{-1}\cal B_{\sigma^{n-(i-1)} \omega})-\sum_{i=1}^{j}\E_{\mu_\omega}(\psi_{\sigma^{n-i} \omega}^2\circ F_\omega^{n-i})\big|\Big\|_{L^{q/2}(\mu_\omega)}\\\nonumber
=&\Big\|\max_{1\le j\le n}\big|\sum_{i=1}^{j}\E_{\mu_\omega}\big(\psi_{\sigma^{n-i} \omega}^2\circ F_\omega^{n-i}-\E_{\mu_\omega}(\psi_{\sigma^{n-i} \omega}^2\circ F_\omega^{n-i})\big|( F_{\omega}^{n-(i-1)})^{-1}\cal B_{\sigma^{n-(i-1)} \omega}\big)\big|\Big\|_{L^{q/2}(\mu_\omega)}\\\nonumber
=&\Big\|\max_{1\le j\le n}\big|\sum_{i=1}^{j}\Big[P_{\sigma^{n-i}\omega}\big(\psi_{\sigma^{n-i}\omega}^{2}-\int\psi_{\sigma^{n-i}\omega}^{2}
d\mu_{\sigma^{n-i}\omega}\big)\Big]\circ F_{\omega}^{n-(i-1)}\big|
\Big\|_{L^{q/2}(\mu_\omega)}\\\nonumber
=&\Big\|\max_{1\le j\le n}\big|\sum_{i=1}^{j}\breve{\phi}_{\sigma^{n-i}\omega}\circ F_{\omega}^{n-i}\big|
\Big\|_{L^{q/2}(\mu_\omega)}\\\nonumber
\le&C_{\omega,\psi,q} n^{\frac{1}{2}},
\end{align*}
where the last inequality follows from Corollary~\ref{smb}. Hence
\begin{align}\label{lk}
&\nonumber\Big\|\sup_{t\in[0,1]}|l-k|\Big\|_{L^{q/2}(\mu_\omega)}\\
\le& \nonumber C\max_{1\le i\le n}\int\psi_{\sigma^{n-i}\omega}^2\circ  F_\omega^{n-i} \rmd\mu_\omega+C\Big\|\max_{1\le j\le n}\big|V_{n,j}^\omega\cdot \alpha_n^2(\omega)-\alpha_j^2(\omega)\big|\Big\|_{L^{q/2}(\mu_\omega)}\\
= &O_\omega(n^{\frac{2}{q}})+O_\omega(n^{\frac{1}{2}})=O_\omega(n^{\frac{1}{2}}).
\end{align}
It follows from Proposition~\ref{Zn} and \eqref{lk} that
\begin{align*}
&\Big\|\sup_{t\in[0,1]}|\widehat M_n^\omega(t)-X_n^\omega(t)|\Big\|_{L^{q/4}(\mu_\omega)}\\
\le&\frac{1}{\eta_n(\omega)}\Big\|\sup_{t\in [0,1]}\big|\sum_{i=k+1}^{l}\psi_{\sigma^{n-i} \omega}\circ F_\omega^{n-i}\big|\Big\|_{L^{q/4}(\mu_\omega)}\\
\le&\frac{1}{\eta_n(\omega)}\Big\|Z_n^\omega\big(\sup_{t\in [0,1]}|l-k|n^{-1/2}+3\big)\Big\|_{L^{q/4}(\mu_\omega)}\\
\le&\frac{1}{\eta_n(\omega)}\big\|Z_n^\omega\big\|_{L^{q/2}(\mu_\omega)}\big(n^{-1/2}\big\|\sup_{t\in [0,1]}|l-k|\big\|_{L^{q/2}(\mu_\omega)}+3\big)\\
\le &C_{\omega,\psi,q} n^{-\frac{1}{4}+\frac{1}{2q}}.
\end{align*}

On the other hand, we note that
\begin{align*}
&g\circ \overline M_{n}^\omega(t)=\overline M_n^\omega(1)-\overline M_n^\omega(1-t)\\
&=\frac{1}{\eta_n(\omega)}\sum_{i=0}^{n-1}\psi_{\sigma^i\omega}\circ  F_{\omega}^i-\frac{1}{\eta_n(\omega)}\sum_{i=0}^{r_n^\omega(1-t)-1}\psi_{\sigma^i\omega}\circ  F_{\omega}^i+E_n^\omega(t)\\
&=\frac{1}{\eta_n(\omega)}\sum_{i=r_n^\omega(1-t)}^{n-1}\psi_{\sigma^i\omega}\circ  F_{\omega}^i+E_n^\omega(t)\\
&=\frac{1}{\eta_n(\omega)}\sum_{i=1}^{n-r_n^\omega(1-t)}\psi_{\sigma^{n-i}\omega}\circ F_{\omega}^{n-i}+E_n^\omega(t),
\end{align*}
where $\|\sup_{t\in[0,1]}E_n^\omega(t)\|_{L^{q}(\mu_\omega)}\le \eta_n^{-1}(\omega)\|\max_{0\le i\le n-1}|\psi_{\sigma^i\omega}\circ  F_{\omega}^i|\|_{L^{q}(\mu_\omega)}\le Cn^{-\frac{1}{2}+\frac{1}{q}}$.

Comparing $n-r_n^\omega(1-t)$ with $l=l_{n,t}^\omega$, we can find that
\[\eta_{r_n^\omega(1-t)-1}^2(\omega)<(1-t)\eta_n^2(\omega)\le\eta_{r_n^\omega(1-t)}^2(\omega).\]
Since $\eta_n^2(\omega)=\alpha_n^2(\omega)$, we have
\[\alpha_n^2(\omega)-\alpha_{n-r_n^\omega(1-t)+1}^2(\omega)<(1-t)\alpha_n^2(\omega)\le \alpha_n^2(\omega)-\alpha_{n-r_n^\omega(1-t)}^2(\omega),\]
i.e.
\[\alpha_{n-r_n^\omega(1-t)}^2(\omega)\le t\alpha_n^2(\omega)<\alpha_{n-r_n^\omega(1-t)+1}^2(\omega).\]
By the definition of $l_{n,t}^\omega$, we also have
$\alpha_{l}^2(\omega)\le t\alpha_n^2(\omega)<\alpha_{l+1}^2$. So $l_{n,t}^\omega=n-r_n^\omega(1-t)$.
Hence
\[
\Big\|\sup_{t\in[0,1]}|g\circ \overline M_{n}^\omega(t)-\widehat M_n^\omega(t)|\Big\|_{L^{q}(\mu_\omega)}\le C\frac{1}{\eta_n(\omega)}\Big\|\max_{0\le i\le n-1 }|\psi_{\sigma^i\omega}\circ F_\omega^i|\Big\|_{L^{q}(\mu_\omega)}\le Cn^{-\frac{1}{2}+\frac{1}{q}}.
\]

Combining the above estimates, by the definition of Wasserstein distance, we obtain that for a.e. $\omega\in \Omega$ and $n\ge1$,
\begin{align*}
\mathcal{W}_{\frac{q}{4}}(g\circ \overline M_{n}^\omega, X_{n}^\omega)&\le \mathcal{W}_{q}(g\circ \overline M_{n}^\omega,
\widehat M_n^\omega)+\mathcal{W}_{\frac{q}{4}}(\widehat M_n^\omega, X_{n}^\omega)\\
&\le Cn^{-\frac{1}{2}+\frac{1}{q}}+Cn^{-\frac{1}{4}+\frac{1}{2q}}\le Cn^{-\frac{1}{4}+\frac{1}{2q}}.
\end{align*}

\end{proof}

\begin{proof}[Proof of Theorem \ref{thnon2}]
Recall that $g:C[0,1]\rightarrow C[0,1]$ is a continuous transformation defined by $g(u)(t)=u(1)-u(1-t)$. we note that $g\circ g= Id$ and $g$ is the Lipschitz with Lip$g$ $\le 2$. It follows from the Lipschitz mapping theorem that for a.e. $\omega\in \Omega$,
\[
\mathcal{W}_{\frac{q}{4}}(\overline M_{n}^\omega,B)= \mathcal{W}_{\frac{q}{4}}(g(g\circ \overline M_{n}^\omega),g(g\circ B))\le 2\mathcal{W}_{\frac{q}{4}}(g\circ \overline M_{n}^\omega,g\circ B).
\]
Since $g(B)=_d B$, by Lemmas \ref{xnw} and \ref{wnxn}, for a.e. $\omega\in \Omega$ and $q\ge 4$, we have
\begin{align*}
&\mathcal{W}_{\frac{q}{4}}(g\circ \overline M_{n}^\omega,g\circ B)\le\mathcal{W}_{\frac{q}{4}}(g\circ \overline M_{n}^\omega,   X_n^\omega)+\mathcal{W}_{\frac{q}{2}}( X_n^\omega, B)\\
&\le C n^{-\frac{1}{4}+\frac{1}{2q}}+C n^{-\frac{1}{4}+\varepsilon}\le C n^{-\frac{1}{4}+\frac{1}{2q}},
\end{align*}
where the last inequality holds because $\varepsilon>0$ can be taken arbitrarily small.

Finally, by Lemma~\ref{wmn}, we conclude that
\begin{align*}
\mathcal{W}_{\frac{q}{4}}(\overline W_{n}^\omega,B)&\le \mathcal{W}_{q}(\overline W_{n}^\omega,\overline M_n^\omega)+\mathcal{W}_{\frac{q}{4}}(\overline M_{n}^\omega,B)\\
&\le Cn^{-\frac{1}{4}+\frac{1}{2q}}+Cn^{-\frac{1}{4}+\frac{1}{2q}}\le Cn^{-\frac{1}{4}+\frac{1}{2q}}.
\end{align*}
\end{proof}

\section{Examples}

In this section, we introduce the following random dynamical systems (RDS) as concrete examples to which
Theorems~\ref{thnon1} and \ref{thnon2} can be applied.
First, we show that Theorems~\ref{thnon1} and \ref{thnon2} can be passed from the RYT to the original RDS.
Let $(X,m,d)$ be a compact manifold with a Riemannian volume $m$ and a Riemannian distance $d$.
Let $f_\omega:X\to X$ be a family of non-singular transformations w.r.t. the Riemannian volume $m$ and define the random orbits $f_\omega^n=f_{\sigma^{n-1}\omega}\circ \cdots\circ f_{\sigma\omega}\circ f_\omega$.
For a.e. $\omega\in \Omega$ and $(x,l)\in \Delta_\omega$, we define a projection $\pi_\omega: \Delta_\omega\to X$ as $\pi_\omega(x,l):=f_{\sigma^{-l}\omega}^{l}(x)$. Then $\pi_\omega$ is a semiconjugacy,
i.e., $f_\omega\circ\pi_\omega=\pi_{\sigma\omega}\circ F_\omega$.
Since $\mu_\omega$ is an absolutely continuous  $F_\omega$-equivariant probability measures on $\Delta_\omega$, then $\nu_\omega:=(\pi_\omega)_*\mu_\omega$ is a family of absolutely continuous $f_\omega$-equivariant probability measures on $\Omega\times X$.

For any H\"older continuous function $\varphi$ on $X$ with a H\"older exponent $\eta\in (0,1]$, define
\[\widetilde\varphi_\omega=\varphi-\int_X \varphi \,\rmd\nu_\omega, \quad \widetilde\varphi(\omega,\cdot)=\widetilde\varphi_\omega(\cdot).\]
The lifted function $\phi_\omega=\widetilde\varphi_\omega\circ \pi_\omega$ with $\phi(\omega,\cdot)=\phi_\omega(\cdot)$ satisfies
\[
\|\phi_\omega\|_{L^\infty(\Delta_\omega)}\le \max_{x\in X}|\varphi(x)|, \quad \int_{\Delta_\omega} \phi_\omega \,\rmd\mu_\omega=0.
\]

\begin{prop}
Assume that there are constants $\beta\in (0,1)$, $C\ge 1$ and a function $K\ge 1$ satisfying \eqref{k} such that  for a.e. $\omega\in \Omega$, any $\Lambda_j(\omega)\in \mathcal P_\omega$, any $x,y\in \Lambda_j(\omega)$, and
$0\le k\le R_\omega|_{\Lambda_j(\omega)}$, the RDS satisfies
\begin{align}\label{exp}
d(f_\omega^{R_\omega}(x),f_\omega^{R_\omega}(y))\ge \beta^{-1}d(x,y),
\end{align}
\begin{align}\label{ktr}
d(f_\omega^k(x),f_\omega^k(y))\le CK_{\sigma^k\omega}d(f_\omega^{R_\omega}(x),f_\omega^{R_\omega}(y)).
\end{align}
Then $\phi\in \mathcal F_{\beta^\eta}^{K^\eta}$ with a Lipschitz constant $C_\varphi C^\eta\big(\sup_{x,y\in X}d(x,y)\big)^\eta (\beta^\eta)^{-1}$.
\end{prop}
\begin{proof}
Let $A=\sup_{x,y\in X}d(x,y)$. By \eqref{exp}, for any $(x,l), (y,l)\in \Delta_\omega$ with $s_\omega((x,l), (y,l))=n$, we have
\[
d(x,y)\le \beta d(f_\omega^{R_\omega}(x),f_\omega^{R_\omega}(y))\le\cdots\le \beta^nd(f_\omega^{R_\omega^n}(x),f_\omega^{R_\omega^n}(y))\le \beta^n A.
\]
Thus, for any $(x,l), (y,l)\in \Delta_\omega$ with $s_\omega((x,l), (y,l))=n$,
\begin{align*}
&|\phi_\omega(x,l)-\phi_\omega(y,l)|=|\widetilde\varphi_\omega(f_{\sigma^{-l}\omega}^{l}x)-\widetilde\varphi_\omega(f_{\sigma^{-l}\omega}^{l}y)|\\
\le& C_\varphi d(f_{\sigma^{-l}\omega}^{l}x,f_{\sigma^{-l}\omega}^{l}y)^\eta
\le C_\varphi C^\eta K_\omega^\eta d(f_{\sigma^{-l}\omega}^{R_{\sigma^{-l}\omega}}x,f_{\sigma^{-l}\omega}^{R_{\sigma^{-l}\omega}}y)^\eta\\
\le& C_\varphi C^\eta K_\omega^\eta\big(\beta^{s_\omega((x,l), (y,l))-1}A\big)^\eta\\
=&C_\varphi C^\eta A^\eta (\beta^\eta)^{-1} K_\omega^\eta( \beta^\eta)^{s_\omega((x,l), (y,l))}.
\end{align*}
\end{proof}

Now, considering the original RDS that satisfies \eqref{exp}, \eqref{ktr} and admits a RYT satisfying the conditions (P1)-(P8), we show that Theorem~\ref{thnon1} and \ref{thnon2} holds for the RDS. Note that
\[
\int(\sum_{i=0}^{n-1}\widetilde\varphi_{\sigma^i\omega}\circ f_\omega^i)^2\rmd\nu_\omega=\int(\sum_{i=0}^{n-1}\phi_{\sigma^i\omega}\circ F_\omega^i)^2\rmd\mu_\omega.
\]
Consider the continuous process $W_{n}^{\omega}\in C[0,1]$ defined by
\[
W_{n}^{\omega}(t)=\frac{1}{\sqrt n}\Big[\sum_{i=0}^{[nt]-1}\widetilde\varphi_{\sigma^i\omega}\circ f_{\omega}^{i}+(nt-[nt])\widetilde\varphi_{\sigma^{[nt]}\omega}\circ f_{\omega}^{[nt]}\Big],\quad t\in[0,1].
\]
Hence,\\
(1) There is  a constant $\Sigma^2\ge 0$ such that
\[
\lim_{n\to\infty}\frac{\int(\sum_{i=0}^{n-1}\widetilde\varphi_{\sigma^i\omega}\circ f_\omega^i)^2\rmd\nu_\omega}{n}=\lim_{n\to\infty}\frac{\int(\sum_{i=0}^{n-1}\phi_{\sigma^i\omega}\circ F_\omega^i)^2\rmd\mu_\omega}{n}=\Sigma^2.
\]
(2) If $\Sigma^2=0$,  then $\widetilde\varphi$ is a coboundary, i.e. there exists $g\in L^q(\Omega\times X,\nu)$ such that for a.e. $\omega\in \Omega$,
\[\widetilde\varphi_{\omega}=g_{\sigma\omega}\circ f_\omega-g_\omega\quad \nu_\omega\hbox{-a.s.}\]
Namely,
\[
\widetilde\varphi=g\circ T-g,\quad \nu\hbox{-a.s.}
\]
Here $T(\omega,x)=(\sigma\omega,f_\omega(x))$ is the skew product.
It follows from the same arguments on projection for coboundary in \cite[Theorem~6.1]{Su22}, which used Theorem~1.1 in \cite{L96} to apply to the skew product $T$ and the observable $\widetilde\varphi$.\\
(3) If $\Sigma^2>0$, then for a.e. $\omega\in \Omega$, $W_n^\omega\to_w W$ in $C[0,1]$, where $W$ is a Brownian motion with variance $\Sigma^2>0$. Indeed,
$W_n^\omega=_d W_n^\omega\circ\pi_\omega.$

Similarly, consider
\begin{align*}
\widetilde W_{n}^{\omega}(t):=\frac{1}{\Sigma_n(\omega)}\Big[\sum_{i=0}^{N_n^\omega(t)-1}\widetilde\varphi_{\sigma^i\omega}\circ f_{\omega}^{i}+\frac{t\Sigma_n^2-\Sigma_{N_n^{\omega}(t)-1}^2}{\Sigma_{N_n^{\omega}(t)}^2-\Sigma_{N_n^\omega(t)-1}^2}\widetilde\varphi_{\sigma^{N_n^\omega(t)}\omega}\circ f_{\omega}^{N_n^\omega(t)}\Big],\quad t\in[0,1],
\end{align*}
where $N_n^\omega$ is defined as \eqref{Nt}. Then Theorem~\ref{thnon2} can be applied to the RDS, that is,
\[
\mathcal{W}_{\frac{q}{4}}(\widetilde W_{n}^{\omega},B)=\mathcal{W}_{\frac{q}{4}}(\widetilde W_{n}^{\omega}\circ\pi_\omega,B)=O(n^{-\frac{1}{4}+\frac{1}{2q}}).
\]

We now apply these results to specific examples. Our goal is to verify that the following RDS satisfy \eqref{exp}, \eqref{ktr},
and admit a RYT satisfying the conditions (P1)--(P8).
Specifically, we will consider i.i.d. perturbations of Viana maps, i.i.d. perturbations of  interval maps of  with a neutral fixed point,
and  small random perturbations of Anosov diffeomorphisms with an ergodic driving system.
It is well established that these systems admit a RYT. Upon construction of the RYT, conditions (P1)--(P3), (P5) and \eqref{exp} are satisfied.
Thus it remains to verify conditions (P4), (P6)--(P8) and \eqref{ktr}.

\begin{exam}[i.i.d. translations of Viana maps]
Consider the random perturbations of Viana maps \cite{AA03, AV13}. It was pointed out in Section~5.2.2 of \cite{AV13} that the return times have stretched exponential tails. So the property (P4) holds. By Remark~\ref{com}, we know that (P8) holds without (P6) and (P7). The condition \eqref{ktr} follows from Proposition~4.9 of \cite{AV13}. Hence, for a.e. $\omega\in \Omega$ and for any  $\delta>0$, $q\ge 4$, we have $\mathcal{W}_{\frac{q}{4}}(\widetilde W_{n}^{\omega},B)=O(n^{-\frac{1}{4}+\delta})$ for all
$n\ge 1$.
\end{exam}

\begin{exam}[i.i.d. perturbations of  interval maps with a neutral fixed point]
Our setting is same as that in \cite{BBR19}, where the authors considered a random family of the intermittent  maps introduced in \cite{LSV} as an application.
They constructed a random Young tower and verified conditions (P4), (P6)--(P7), which in turn implies that (P8) holds;
see Section 5 of \cite{BBR19} for details.
Since the derivative of such maps is bounded below by $1$, it follows that
\[
d(f_\omega^k(x),f_\omega^k(y))\le d(f_\omega^{R_\omega}(x),f_\omega^{R_\omega}(y)),
\]
and thus condition \eqref{ktr} is satisfied.
Let $\Omega=[\alpha_0,\alpha_1]^{\Z}$, where $0<\alpha_0<1/5$, and $\alpha_1<1$.
Let $\delta$ be small s.t. $q=\frac{(\frac{1}{\alpha_0}-1-\alpha)\alpha}{\alpha+1}\ge4$.
Then for a.e. $\omega\in \Omega$ and any H\"older function $\varphi$, we obtain the Wasserstein convergence rate
$\mathcal{W}_{\frac{q}{4}}(\widetilde W_{n}^{\omega},B)=O(n^{-\frac{1}{4}+\frac{1}{2q}})$ for all $n\geq 1$.
This result also applies to the examples discussed in \cite{MR24}.
\end{exam}

\begin{exam}[small random perturbations of Anosov diffeomorphisms with an ergodic driving system]
Let $X$ be a smooth compact connected Riemannian manifold of finite dimension and let $T$ be a topologically transitive Anosov map of
class $\mathcal C^{r+1}$ with $r>2$.
Consider a small neighborhood $\mathcal U$ of $T$ in the  $\mathcal C^{r+1}$ topology, and
let $\sigma:\Omega\to \Omega$ be an ergodic automorphism of the probability space $(\Omega,\mathbb P)$.
It follows from \cite{ABR22} that the corresponding RDS admits a random hyperbolic tower, as defined in \cite{ABR22}, with exponential tails.
From this structure, we obtain a quotient random tower that satisfies (P1)-(P5).
Based on the proofs of Theorems \ref{thnon1} and \ref{thnon2}, we observe that (P8) is crucial.
If (P8) holds, the results still follow even in the absence of conditions (P6) and (P7).
In \cite{DH21}, the authors proved (P8) with $C_\omega=1$ using spectral methods.
Moreover, condition (7.2) is a direct consequence of \cite[Proposition~3.3]{ABR22}.
Hence, for a.e. $\omega\in \Omega$ and for any  $\delta>0$, $q\ge 4$, we have $\mathcal{W}_{\frac{q}{4}}(\widetilde W_{n}^{\omega},B)=O(n^{-\frac{1}{4}+\delta})$ for all
$n\ge 1$.
\end{exam}
\appendix

\section{}

\subsection{Quenched invariance principle for reverse martingale}

Consider the random dynamical systems $(\Omega, \P, \sigma, (\Delta_\omega)_{\omega\in\Omega}, (\mu_\omega)_{\omega\in\Omega},(F_\omega)_{\omega\in\Omega})$. Let $q\ge4$. Suppose that $\psi\in L^q(\Delta, \mu)$ with $\psi(\omega,\cdot)=\psi_\omega(\cdot)$, $\int \psi_\omega \rmd\mu_\omega=0$ and $\psi_\omega\in {\rm ker} P_\omega$.

For a.e. $\omega\in \Omega$ and $t\in [0,1]$, consider
\[M_n^\omega(t)=\frac{1}{\sqrt n}\Big(\sum_{i=0}^{[nt]-1}\psi_{\sigma^i\omega}\circ F_\omega^i+(nt-[nt])\psi_{\sigma^{[nt]}\omega}\circ F_\omega^{[nt]}\Big),\quad t\in [0,1].\]

\begin{thm}
Suppose that there exists a constant $\Sigma^2>0$ such that for a.e. $\omega\in \Omega$ and each $t\in[0,1]$,
\[
\frac{1}{n}\sum_{i=0}^{[nt]-1}[P_{\sigma^i\omega}(\psi_{\sigma^i\omega}^2)]\circ F_\omega^{i+1}\to_{w}t\Sigma^2, \quad \hbox{~as~}n\to \infty.
\]
Then for a.e. $\omega\in \Omega$, $M_n^\omega\to_{w}W$ in $C[0,1]$, where $W$ is a Brownian motion with variance $\Sigma^2$.
\end{thm}
\begin{proof}
{\bf Convergence of finite-dimensional distributions.} For a.e. $\omega\in \Omega$, fix $0<t_1<t_2<\cdots<t_k<1$, $c_1,c_2,\cdots,c_k\in \R$.
Define
\[
Z_n^\omega=\sum_{l=1}^{k}c_l(M_n^\omega(t_l)-M_n^\omega(t_{l-1})), \quad Z=\sum_{l=1}^{k}c_l(W(t_l)-W(t_{l-1})).
\]
We aim to show that $Z_n^\omega\to_d Z$. Now, $Z=_dN(0,V)$, where $V=\sum_{l=1}^{k}c_l\Sigma^2(t_l-t_{l-1})$.
We can write
\[
Z_n^\omega=\frac{1}{\sqrt n}\sum_{l=1}^{k}c_l\sum_{i=[nt_{l-1}]}^{[nt_l]-1}\psi_{\sigma^i\omega}\circ F_\omega^i+E_n^\omega=\sum_{j=1}^{[nt_k]}X_{n,j}^\omega+E_n^\omega,
\]
where $X_{n,j}^\omega=n^{-1/2}d_j\psi_{\sigma^{[nt_k]-j}\omega}\circ F_\omega^{[nt_k]-j}$ for appropriate choice of $d_j\in\{c_1,c_2,\cdots,c_k\}$ and  $\|E_n^\omega\|_{L^q(\mu_\omega)}=O_\omega(n^{-1/2+1/q})$. Then $\{X_{n,j}^\omega\}_{1\le j\le [nt_k]}$ is a martingale difference array w.r.t. the filtration $\mathcal G_{n,j}^\omega=(F_\omega^{[nt_k]-j})^{-1}\mathcal B_{\sigma^{[nt_k]-j}\omega}$.

We now apply a CLT for martingale difference arrays. To show $Z_n^\omega\to_d N(0,V)$, by \cite[Theorem~18.1]{B99a}, it suffices to show that\\
(A1) $\sum_{j=1}^{[nt_k]}\E_{\mu_\omega}((X_{n,j}^\omega)^2|\mathcal G_{n,j-1}^\omega)\to_w V$ as $n\to\infty$.\\
(A2) $\lim_{n\to\infty}\sum_{j=1}^{[nt_k]}\E_{\mu_\omega}((X_{n,j}^\omega)^21_{\{|X_{n,j}^\omega|\ge \eps\}})=0$ for all $\eps>0$.\\

By \eqref{cvm}, $\E_{\mu_\omega}((X_{n,j}^\omega)^2|\mathcal G_{n,j-1}^\omega)=n^{-1}\big[P_{\sigma^{[nt_k]-j}\omega}(d_j^2\psi_{\sigma^{[nt_k]-j}\omega}^2)\big]\circ F_\omega^{[nt_k]-j+1}$. Hence
\begin{align*}
&\sum_{j=1}^{[nt_k]}\E_{\mu_\omega}((X_{n,j}^\omega)^2|\mathcal G_{n,j-1}^\omega)\\
=&n^{-1}\sum_{j=1}^{[nt_k]}\big[P_{\sigma^{[nt_k]-j}\omega}(d_j^2\psi_{\sigma^{[nt_k]-j}\omega}^2)\big]\circ F_\omega^{[nt_k]-j+1}\\
=&n^{-1}\sum_{l=1}^{k}\sum_{j=[nt_{l-1}]}^{[nt_l]-1}c_l^2\big[P_{\sigma^{j}\omega}\psi_{\sigma^{j}\omega}^2\big]\circ F_\omega^{j+1}\\
=&\sum_{l=1}^{k}c_l^2\Big[\frac{1}{n}\sum_{j=0}^{[nt_l]-1}\big[P_{\sigma^{j}\omega}\psi_{\sigma^{j}\omega}^2\big]\circ F_\omega^{j+1}-\frac{1}{n}\sum_{j=0}^{[nt_{l-1}]-1}\big[P_{\sigma^{j}\omega}\psi_{\sigma^{j}\omega}^2\big]\circ F_\omega^{j+1}\Big]\\
\to_w&\sum_{l=1}^{k}c_l^2(t_l\Sigma^2-t_{l-1}\Sigma^2)=V,\quad \hbox{~as~}n\to\infty.
\end{align*}
Next, we set $K=\max\{c_1,c_2,\cdots,c_k\}$. Then by the H\"older and Chebyshev inequality, we have
\begin{align*}
&\sum_{j=1}^{[nt_k]}\E_{\mu_\omega}((X_{n,j}^\omega)^21_{\{|X_{n,j}^\omega|\ge \eps\}})\\
\le&\frac{K^2}{n}\sum_{j=1}^{[nt_k]}\big\|\psi_{\sigma^{[nt_k]-j}\omega}\circ F_\omega^{[nt_k]-j}\big\|_{L^4(\mu_\omega)}^2\mu_\omega\Big(|\psi_{\sigma^{[nt_k]-j}\omega}\circ F_\omega^{[nt_k]-j}|>\frac{\eps\sqrt n}{K}\Big)^{1/2}\\
\le&\frac{K^2}{n}\sum_{j=1}^{[nt_k]}\big\|\psi_{\sigma^{[nt_k]-j}\omega}\circ F_\omega^{[nt_k]-j}\big\|_{L^4(\mu_\omega)}^2\Big(\frac{K^4}{(\eps{\sqrt n})^4}\E_{\mu_\omega}\big|\psi_{\sigma^{[nt_k]-j}\omega}\circ F_\omega^{[nt_k]-j}\big|^4\Big)^{1/2}\\
=&\frac{K^{4}}{\eps^{2}n^{2}}\sum_{j=1}^{[nt_k]}\big\|\psi_{\sigma^{[nt_k]-j}\omega}\circ F_\omega^{[nt_k]-j}\big\|_{L^4(\mu_\omega)}^{4}\to 0
\end{align*}
by Birkhoff's ergodic theorem. These show that the finite-distributions converge.

{\bf Tightness.} Since $M_n^\omega(0)=0$, by \cite[Theorem~7.3]{B99a}, proving tightness of $M_n^\omega$ is equivalent to showing that
\begin{align}\label{tight}
\lim_{\delta\to 0}\limsup_{n\to\infty}\mu_\omega\Big(\sup_{0\le s,t\le 1\atop|s-t|\le\delta}|M_n^\omega(t)-M_n^\omega(s)|>\eps\Big)=0\hbox{~for~every~} \eps>0.
\end{align}
Define
\[\widetilde M_n^\omega(t)=\frac{1}{\sqrt n}\Big(\sum_{i=1}^{[nt]}\psi_{\sigma^{n-i}\omega}\circ F_\omega^{n-i}+(nt-[nt])\psi_{\sigma^{n-([nt]+1)}\omega}\circ F_\omega^{n-([nt]+1)}\Big).\]
We claim that the hypothesis in \cite[Theorem~18.2]{B99a} are satisfied. Hence in particular $\widetilde M_n^\omega$ is tight. So \eqref{tight} holds with $M_n^\omega$ replaced by $\widetilde M_n^\omega$.

Let $u_{n,t}\in[0,1]$ be such that $[nu_{n,t}]=n-[nt]$.
Then
\begin{align*}
&\sup_{0\le s,t\le 1\atop|s-t|\le\delta}|M_n^\omega(t)-M_n^\omega(s)|\\
\le&\sup_{0\le s,t\le 1\atop|s-t|\le\delta}|\widetilde M_n^\omega(u_{n,s})-\widetilde M_n^\omega(u_{n,t})|+\frac{2}{\sqrt n}\max_{1\le k\le n}|\psi_{\sigma^{k}\omega}\circ F_\omega^k|\\
=&\sup_{0\le s,t\le 1\atop|s-t|\le\frac{2}{n}+\delta}|\widetilde M_n^\omega(t)-\widetilde M_n^\omega(s)|+\frac{2}{\sqrt n}\max_{1\le k\le n}|\psi_{\sigma^{k}\omega}\circ F_\omega^k|.
\end{align*}
Since
\[
\mu_\omega\big(\max_{1\le k\le n}|\psi_{\sigma^{k}\omega}\circ F_\omega^k|>\eps\sqrt n)=\frac{1}{(\eps\sqrt n)^4}\E_{\mu_\omega}\max_{1\le k\le n}|\psi_{\sigma^{k}\omega}\circ F_\omega^k|^4\to 0
\]
by Birkhoff's ergodic theorem. Hence the tightness of $M_n^\omega$ is obtained.

In the following, we need to verify the hypothesis. Consider the martingale difference arrays $\{X_{n,j}^\omega, \mathcal G_{n,j}^\omega\}$ with $X_{n,j}^\omega=n^{-1/2}\psi_{\sigma^{n-j}\omega}\circ F_\omega^{n-j}$ and $\mathcal G_{n,j}^\omega=(F_\omega^{n-j})^{-1}\mathcal B_{\sigma^{n-j}\omega}$.
It suffices to show that for each $t\in[0,1]$,\\
(A3) $\sum_{j=1}^{[nt_k]}\E_{\mu_\omega}((X_{n,j}^\omega)^2|\mathcal G_{n,j-1}^\omega)\to_w V$ as $n\to\infty$.\\
(A4) $\lim_{n\to\infty}\sum_{j=1}^{[nt_k]}\E_{\mu_\omega}((X_{n,j}^\omega)^21_{\{|X_{n,j}^\omega|\ge \eps\}})=0$ for all $\eps>0$.\\
They are proved in the same way as (A1) and (A2) above; we omit it.
\end{proof}

\section*{Acknowledgements}
Zhenxin Liu is supported by NSFC (Grant 11925102) and Liaoning Revitalization Talents Program (Grant XLYC2202042). Sandro Vaienti thanks the CNRS International Research Laboratory IRL 2032 FAMSI, France Australia, at the Australian National University in Canberra, where this work was completed. Zhe Wang would like to acknowledge support from China Scholarship Council and the warm hospitality of Aix-Marseille University at the Institut de Mathématique de Marseille.


\end{document}